\documentclass[11pt]{amsart}

\usepackage{amsmath,amssymb,latexsym,soul,cite,mathrsfs}

\usepackage{color,enumitem,graphicx}
\usepackage[colorlinks=true,urlcolor=blue,
citecolor=red,linkcolor=blue,linktocpage,pdfpagelabels,
bookmarksnumbered,bookmarksopen]{hyperref}
\usepackage[english]{babel}

\usepackage[left=2.9cm,right=2.9cm,top=2.8cm,bottom=2.8cm]{geometry}
\usepackage[hyperpageref]{backref}

\usepackage[colorinlistoftodos]{todonotes}
\makeatletter
\providecommand\@dotsep{5}
\def\listtodoname{List of Todos}
\def\listoftodos{\@starttoc{tdo}\listtodoname}
\makeatother

\numberwithin{equation}{section}

\newtheorem{theorem}{Theorem}[section]
\newtheorem{proposition}[theorem]{Proposition}
\newtheorem{lemma}[theorem]{Lemma}
\newtheorem{corollary}[theorem]{Corollary}

\newcommand\R{\mathbb R}
\newcommand\N{\mathbb N}

\begin{document}

\title[exterior domains with nonlocal Neumann condition]
{Multiplicity of solutions for a class of fractional elliptic problem with critical exponential growth and nonlocal Neumann condition}

\author{Claudianor O. Alves}
\author{C\'esar E. Torres Ledesma}

\address[Claudianor O. Alves]{\newline\indent Unidade Acad\^emica de Matem\'atica
\newline\indent 
Universidade Federal de Campina Grande,
\newline\indent
58429-970, Campina Grande - PB - Brazil}
\email{\href{mailto:coalves@mat.ufcg.edu.br}{coalves@mat.ufcg.edu.br}}

\address[C\'esar E. Torres Ledesma]
{\newline\indent Departamento de Matem\'aticas
\newline\indent 
Universidad Nacional de Trujillo
\newline\indent
Av. Juan Pablo II s/n. Trujillo-Per\'u}
\email{\href{ctl\_576@yahoo.es}{ctl\_576@yahoo.es}}

\pretolerance10000


\begin{abstract}
\noindent In this paper we consider the existence and multiplicity of weak solutions for the following class of fractional elliptic problem 
\begin{equation}\label{00}
\left\{\begin{aligned}
(-\Delta)^{\frac{1}{2}}u + u &= Q(x)f(u)\;\;\mbox{in}\;\;\R \setminus (a,b)\\
\mathcal{N}_{1/2}u(x) &= 0\;\;\mbox{in}\;\;(a,b),
\end{aligned}
\right.
\end{equation} 
where $a,b\in \R$ with $a<b$, $(-\Delta)^{\frac{1}{2}}$ denotes the fractional Laplacian operator and $\mathcal{N}_s$ is the nonlocal operator that describes the Neumann boundary condition, which is given by 
$$
\mathcal{N}_{1/2}u(x) = \frac{1}{\pi} \int_{\R\setminus (a,b)} \frac{u(x) - u(y)}{|x-y|^{2}}dy,\;\;x\in [a,b]. 
$$    

\end{abstract}

\subjclass[2010]{Primary 35A15; Secondary 35J60, 34B10} 
\keywords{Variational methods, Nonlinear elliptic equations, Nonlocal Problems }

\maketitle

\section{Introduction}

In this paper, we deal with the existence and multiplicity of weak solutions for the following class of fractional problem with nonlocal Neumann boundary condition
$$
\left\{
\begin{aligned}
(-\Delta)^{\frac{1}{2}}u + u &= Q(x)f(u)\;\;\mbox{in}\;\;\R \setminus (a,b)\\
\mathcal{N}_{1/2}u(x) &= 0\;\;\mbox{in}\;\;(a,b),\end{aligned}
\right.
\eqno{(P)}
$$
where $a,b\in \R$ with $a<b$, $\mathcal{N}_{1/2}$ denotes the non local normal derivative, defined as   
\begin{equation}\label{I01}
\mathcal{N}_{1/2}u(x) = \frac{1}{\pi} \int_{\R\setminus (a,b)} \frac{u(x) - u(y)}{|x-y|^{2}}dy,\;\;x\in [a,b],
\end{equation}
and  $(-\Delta)^{\frac{1}{2}}$ denotes the fractional Laplacian operator defined as,
\begin{equation}\label{I02}
(-\Delta)^{\frac{1}{2}}u(x) = -\frac{1}{2\pi}\int_{\R}\frac{u(x + y)-u(x-y)-2u(x)}{|y|^{2}}dy,
\end{equation}
and $f:\R \to \R$ is a smooth nonlinearity with an exponential critical growth. One of the main motivation to study this type, via variational methods in the fractional Sobolev space $H^{1/2}(\R)$ has been motivated by an interesting Trudinger-Moser
type inequality due to Lula et al. \cite{ Iu16} ( see also Ozawa \cite{Ozawa} )
\begin{equation}\label{I03}
\sup_{u\in H^{1/2}(\R), \|u\|_{1/2}\leq 1}\int_{\R}\left( e^{\pi |u|^2} - 1\right)dx < \infty.
\end{equation}

In view of (\ref{I03}), we say that $f$ has exponential critical growth at $+\infty$, if there exists $\omega \in (0,\pi]$ and $\alpha_0\in (0,\omega)$, such that 
$$
\lim_{|s|\to \infty} \frac{|f(s)|}{e^{\alpha |s|^2}} = \begin{cases}
0,&\mbox{for all}\;\;\alpha > \alpha_0,\\
+\infty,&\mbox{for all}\;\;\alpha < \alpha_0.
\end{cases}
$$
Recently, partial differential equations involving the fractional Laplacian operator $(-\Delta)^{s}$ with $s \in (0,1)$ has received a special attention, because its arises in a quite natural way in many different contexts, such as, among the others, the thin obstacle problem, optimization, finance, phase transitions, stratified materials, anomalous diffusion, crystal dislocation, soft thin films, semipermeable membranes, flame propagation, conservation laws, ultra-relativistic limits of quantum mechanics, quasi-geostrophic flows, multiple scattering, minimal surfaces, materials science and water waves, for more detail see \cite{Bucurb, Nez12, Dipierrob, Molicab, CP}.

In the last 20 years, there has been a lot of interest in the study of the existence and multiplicity of nodal solutions for nonlinear elliptic problems. Recently, the existence and multiplicity of nodal solutions for the fractional elliptic problem 
\begin{equation}\label{I04}
\left\{
\begin{aligned}
(-\Delta)^su &= f(x,u)\;\;\mbox{in}\;\;\Omega,\\
u&=0\;\;\mbox{in}\;\;\R^N \setminus \Omega,
\end{aligned}
\right.
\end{equation}
where $s\in (0,1)$ and $\Omega\subset \R^N$ is a smooth bounded domain, has been investigated by Chang and Wang \cite{Chang14}, by using the descended flow methods and harmonic extension techniques. Teng, Wang and Wang  \cite{KTKWRW} have prove the existence of nodal solutions for problem (\ref{I04}) by using the constrained minimization methods and adapting some arguments found in \cite{AlvesSouto}. We note that the main difficulties in the study of problem (\ref{I04}) is related to the presence of the fractional Laplacian $(-\Delta)^s$ which is a nonlocal operator. Indeed, the Euler-Lagrange functional associated to the problem (\ref{I04}), that is
$$
J(u) = \frac{C_{N,s}}{4}\iint_{\R^{2N}\setminus \Omega^c\times \Omega^c}\frac{|u(x) - u(y)|^2}{|x-y|^{N+2s}}dy dx - \int_{\Omega}F(x,u(x))dx 
$$
does not satisfy the decompositions 
$$
\begin{aligned}
&J(u) = J(u^+) + J(u^-)\\
&J'(u)u^{\pm} = J'(u^{\pm})u^{\pm},
\end{aligned}
$$
which were fundamental in applying the variational approach developed in \cite{BartschW05}. 

When $N=1$ and $s=\frac{1}{2}$ in (\ref{I05}), only few papers has appeared in literature; see \cite{VA, Alves16, AlvesFG, doO, Souza16, RFEL, Iannizzoto}. Indeed, one of the main difficulty in the study of this class of problems is related to the fact that the embedding $H^{1/2}(\R) \subset L^q(\R, \R)$ is continuous for all $q\in [2, \infty)$ but not in $L^\infty(\R, \R)$; see \cite{ Nez12}. This means that the maximal growth that we have to impose on the nonlinearity $f$ to deal with (\ref{I05}) via variational methods in a suitable subspace of $H^{1/2}(\R,\R)$, is given by $e^{\alpha_0 |u|^2}$ as $|u|\to \infty$ for some $\alpha_0>0$ which is a consequence of the fractional Moser-Trudinger inequality given in (\ref{I03}).

On the other hand, research has been done in recent years for the fractional elliptic problem with nonlocal Neumann condition. We mention the work by Dipierro, Ros-Oton and Valdinoci \cite{SDXREV}, where they established a complete description of the eigenvalues of $(-\Delta)^s$ with zero non local Neumann boundary condition, an existence and uniqueness result for the elliptic problem and the main properties of the fractional heat equation with this type of boundary condition. Chen \cite{Chen}, has considered the fractional Schr\"odinger equation
\begin{equation}\label{I05}
\left\{
\begin{aligned}
\epsilon^{2s}(-\Delta)^su + u &= |u|^{p-1}u\;\;\mbox{in}\;\;\Omega,\\
\mathcal{N}_su &= 0\;\;\mbox{on}\;\;\R^N \setminus \overline{\Omega},
\end{aligned}
\right.
\end{equation}
where $\epsilon >0$, $s\in (0,1)$, $\Omega\subset \R^N$ be a smooth bounded domain, $p\in (1, \frac{N+2s}{N-2s})$ and 
$$
\mathcal{N}_su(x) = C_{N,s}\int_{\Omega}\frac{u(x) - u(y)}{|x-y|^{N+2s}}dy,\;\;x\in \R^N \setminus \overline{\Omega}.
$$
By using mountain pass theorem, he showed that there exists a non-negative solution $u_\epsilon$ to (\ref{I05}). 

In the local case, i.e. $s=1$, the problem below 
\begin{equation}\label{I06}
\left\{
\begin{aligned}
-\Delta u + u &= Q(x)|u|^{p-1}u,\;\;\mbox{in}\;\;\R^N \setminus \Omega\\
\frac{\partial u}{\partial \eta} &= 0\;\;\mbox{on}\;\;\partial \Omega.
\end{aligned}
\right.
\end{equation}
has received a special attention of some authors. In \cite{VBGC}, Benci and Cerami showed that (\ref{I06}), with $Q\equiv 1$ and Dirichlet condition, has not a ground state solution, that is, there is no a solution of (\ref{I06}) with minima energy. However, Esteban in \cite{ME} proved that the same problem with Neumann condition has a ground state solution. In \cite{DC}, Cao studied the existence of positive solution for problem (\ref{I06}) by supposing that 
\begin{enumerate}
\item[$(Q'_1)$] $Q(x)\geq \tilde{Q}-Ce^{-\nu|x|}|x|^{-m}$ \quad as \quad $|x| \to +\infty$ and $\displaystyle \lim_{|x|\to +\infty}Q(x) = \tilde{Q}>0$,
\end{enumerate}
where $\nu=\frac{2(p+1)}{p-1}$, $m >N-1$ and $C>0$. In the same paper, Cao also studied the existence of solution that changes sign ( nodal solution ), by assuming the following condition on $Q$
\begin{enumerate}
	\item[$(Q'_2)$] $Q(x)\geq \tilde{Q}+Ce^{-\frac{p|x|}{p+1}}|x|^{-m}$ \quad as \quad $|x| \to +\infty$ and $\displaystyle \lim_{|x|\to +\infty}Q(x) = \tilde{Q}>0$,
\end{enumerate}
with $0<m < \frac{N-1}{2}$. In \cite{CAPCEM}, Alves, Carri\~ao and Medeiros showed that the results found in \cite{DC} also hold for the $p$-Laplacian operator and for a larger class of nonlinearity. We also mention the work by Alves \cite{CA}, where problem (\ref{I06}) was considered with a critical growth nonlinearity for $N=2$.  It is very important point out that in all the above mentioned papers the fact that the limit problem in whole $\mathbb{R}^N$ has a ground state solution with exponential decaying is a key point in their arguments, because this type of behavior at infinite works well with conditions $(Q'_1)$ and $(Q'_2)$.

Since we did not find in the literature any paper dealing with the existence of ground state and nodal solutions for problem $(P)$ in exterior domains, motivated by the previous works, we intend in the present paper to prove that $(P)$ has two nontrivial solutions, the first solution is a non-negative ground state solution while the second one is a nodal solution. However, different of  the local case $s=1$, we do not know if the  ground state solution of limit problem in whole $\mathbb{R}$ has an exponential decaying, which brings a lot of difficulties for the nonlocal case. The reader is invited to see that for the existence of nodal solution,  we overcome this difficulty by assuming more a condition on the function $Q$, see condition $(Q_2)$ and Theorem \ref{main2} below. Moreover, we prove a Lions type theorem for exterior domain that is crucial in our approach, see Proposition \ref{GSresult03} in Section 3. The main results of this paper, in some sense, complete the study made in \cite{CA} and \cite{CACT2}, because we are considering a version of those papers for the fractional Laplacian with critical exponential growth. Finally, we would like point out that in \cite{CACT}, Alves, Bisci and Torres have studied $(P)$ in exterior domain with Dirichlet boundary conditions.

In what follows, let us assume that $Q$ is a continuous function that satisfies  
\begin{enumerate}
\item[($Q_1$)] $Q(x) \geq \tilde{Q}>0$ in $\R \setminus (a,b)$ and 
$$
\lim_{|x|\to \infty}Q(x) = \tilde{Q}.
$$
\end{enumerate}
Related to the nonlinearity, we assume that $ f \in C^1(\R, \R)$, odd and verifies the following properties:  
\begin{enumerate}
\item[$(f_1)$] $|f(s)|\leq C e^{\pi |s|^2}$ for all $s\in \R$.
\item[$(f_2)$] There is $\theta >2$ such that 
$$
0< \theta F(s)\leq sf(s)\;\;\mbox{for all}\;\;s\in \R\setminus \{0\}.
$$ 
\item[$(f_3)$] There exists $q>1$ such that 
$$
\limsup_{|s|\to 0} \frac{|f(s)|}{|s|^q}< \infty.
$$ 
\item[$(f_4)$] the function $s\to \frac{f(s)}{s}$ is increasing in $(0, +\infty)$.
\item[$(f_5)$] There are constants $p>q+1$ and $C_p>0$ such that 
$$
f(s)\geq C_p s^{p-1}\;\;\mbox{for all}\;\;s\in [0, \infty),
$$
where 
$$
C_p> \left(\frac{(p-2)2\xi \theta}{(\theta -2)p} \right)^{\frac{p-2}{2}}S_{p}^{p},
$$
$$
S_p = \inf_{u\in H^{1/2}(\R) \setminus \{0\}} \frac{\|u\|_{1/2}}{\left( \int_{\R} \tilde{Q}|u|^{p}dx\right)^{1/p}}
$$
and $\xi$ is a positive constant such that the extension operator $E: H^{1/2}(\R\setminus (a,b))\to H^{1/2}(\R)$ satisfies 
$$
\|Eu\|_{1/2} \leq \xi \|u\|_{H_{\tilde{\Omega}}^{1/2}}\;\;\forall u\in H^{1/2}(\R\setminus \Omega).
$$
For more details see \cite{FDGD}. 
\end{enumerate} 

Now we are in position to state our main result concerning to the existence of ground state solution.
\begin{theorem}\label{main1}
Suppose that $(Q_1)$, $(f_1)-(f_5)$ hold. Then $(P)$ has a ground state solution.  
\end{theorem} 

In order to get a nodal solution, we assume the following additional conditions on $f$:
\begin{enumerate}
\item[($f_6$)] There exists $\sigma \geq 2$ such that 
$$
f'(s)s^2 - f(s)s\geq C|s|^{\sigma},\;\;\forall s\in \R.
$$
\item[($f_7$)] $|f'(s)s| \leq Ce^{\pi s^2}$ for all $s\in \R$ and for some positive constant $C$.
\end{enumerate}

\begin{theorem}\label{main2}
Suppose that $(f_1)-(f_7)$, $(Q_1)$, and that  there are $C>0$, $\gamma > 2p -1$, $R>|a|+|b|+1$ and $\sigma_R \in R$ with $|\sigma_R| > 3R$ such that     
$$
\displaystyle Q(x) - \tilde{Q} \geq CR^{\gamma}, \quad \forall x\in (R+\sigma_R,2R+\sigma_R). \leqno{(Q_2)}
$$ 
Then, there is $R_0>0$ such that $(P)$ has a nodal solution for all $R \geq R_0$.
\end{theorem}

\section{Preliminary Results}
In this section we introduce some function spaces and consider the existence of positive solution of the limit problem 
$$
\left\{
\begin{aligned}
\frac{1}{2}(-\Delta)^{1/2}u + u &= \tilde{Q}f(u)\;\;\mbox{in}\;\;\R, \\
u&\in H^{1/2}(\R).
\end{aligned}
\right.
\eqno{(P_\infty)}
$$

In general in the literature the operator that appears in the limit problem is $(-\Delta)^s$, here we have a new phenomena and we must work with the limit problem involving the operator $\frac{1}{2}(-\Delta)^s$, this is justified because in the energy functional $(P)$ appears the term 
$\frac{1}{4\pi}\iint_{\R^{2}\setminus (a,b)^2}\frac{|u(x) - u(y)|^2}{|x-y|^{2}}dy dx$, see Section 3 for more details, and in this paper we need to do some estimates involving the energy functionals of the Neumann problem and limit problem, in this sense the first part of the two functionals must be quite similar.

We recall that the fractional Sobolev space $H^{1/2}(\R)$ is defined as
$$
H^{1/2}(\R) = \left\{u\in L^2(\R):\;\;\iint_{\R^2} \frac{|u(x) - u(y)^2|}{|x-y|^{2}}dy dx < \infty \right\}
$$
endowed with the norm 
$$
\|u\|_{1/2} = \left( \int_{\R}|u|^2dx + \iint_{\R^2} \frac{|u(x) - u(y)|^2}{|x-y|^{2}}dy dx\right)^{1/2}.
$$
The square root of the Laplacian $(-\Delta)^{1/2}$, of a smooth function $u:\R \to \R$ is defined through Fourier transform by 
$$
\mathcal{F}((-\Delta)^{1/2}u)(\xi) = |\xi| \mathcal{F}(u)(\xi).
$$
By \cite[Proposition 3.6]{Nez12}, we have 
$$
\|(-\Delta)^{1/4}u\|_{L^2(\R)}^{2} := \frac{1}{2\pi}\iint_{\R^2}\frac{|u(x) - u(y)|^2}{|x-y|^{2}}dy dx,\;\;\mbox{for all}\;\;u\in H^{1/2}(\R),
$$  
and the continuous Sobolev embeddings
\begin{equation} \label{Plm01}
H^{1/2}(\R) \hookrightarrow  L^q(\R)\;\;\mbox{for every}\;\,q\in [2,\infty).
\end{equation}

In what follows, we set $\tilde{\Omega} = \R \setminus (a,b)$ and denote by $H_{\tilde{\Omega}}^{1/2}$ the fractional Sobolev space given by
$$
H_{\tilde{\Omega}}^{1/2} = \left\{u:\R \to \R\;\;\mbox{measurable and}\;\; \frac{1}{2\pi}\iint_{\R^2\setminus (a,b)^2}\frac{|u(x) - u(y)|^2}{|x-y|^{2}}dy dx + \int_{\R \setminus (a,b)}|u|^2dx < \infty \right\}
$$ 
endowed with the norm 
$$
\|u\|_{H_{\tilde{\Omega}}^{1/2}} = \left( \frac{1}{2\pi}\iint_{\R^2\setminus (a,b)^2}\frac{|u(x) - u(y)|^2}{|x-y|^{2}}dy dx + \int_{\R \setminus (a,b)} |u|^2dx\right)^{1/2}.
$$
The next lemma will be often used in the present paper
\begin{lemma}\label{embe}
\begin{enumerate}
\item Since $\R^2\setminus (a,b)^2 \subset \R^2$, we have  
$$
\iint_{\R^2\setminus (a,b)^2} \frac{|u(x) - u(y)|^2}{|x-y|^{2}}dy dx \leq \iint_{\R^2} \frac{|u(x) - u(y)|^2}{|x-y|^{2}}dy dx\;\;\mbox{for all $u\in H^{1/2}(\R)$}.
$$
Then the embedding $H^{1/2}(\R) \hookrightarrow H_{\tilde{\Omega}}^{1/2}$ is continuous.
\item Since $\R\setminus (a,b)\times \R \setminus (a,b) \subset \R^2\setminus (a,b)^2$, it follows that
$$
\int_{\R\setminus (a,b)}\int_{\R \setminus (a,b)}\frac{|u(x) - u(y)|^2}{|x-y|^2}dy dx \leq \iint_{\R^2\setminus (a,b)^2}\frac{|u(x) - u(y)|^2}{|x-y|^2}dy dx\;\;\mbox{for all $u\in H_{\tilde{\Omega}}^{1/2}$}.
$$ 
Thus, the embedding $H_{\tilde{\Omega}}^{1/2} \hookrightarrow H^{1/2}(\R\setminus (a,b))$ is continuous.
\item Note that the embedding $H^{1/2}(\R\setminus (a,b))  \hookrightarrow L^q(\R\setminus (a,b))$ is continuous for any $q\in [2, \infty)$. Furthermore, there exists $C_q>0$ such that 
$$
\|u\|_{L^q(\tilde{\Omega})} \leq C_q \|u\|_{H^{1/2}(\tilde{\Omega})}.
$$   
Combining (2) and (3), we can ensure that the embedding $H_{\tilde{\Omega}}^{1/2}  \hookrightarrow L^q(\tilde{\Omega})$ is continuous for any $q\in [2,\infty)$. Moreover, there exists a positive constant $S_q$ such that 
$$
\|u\|_{L^q(\tilde{\Omega})}\leq S_q\|u\|_{H_{\tilde{\Omega}}^{1/2}}\;\;\mbox{for all $u\in H_{\tilde{\Omega}}^{1/2}$}.
$$        
\end{enumerate}
\end{lemma}

\begin{lemma}\label{lmCC}
	\cite{Souza16} If $(u_n)$ is bounded sequence in $H^{1/2}(\R)$ and 
	\begin{equation}\label{04}
	\lim_{n\to \infty}\sup_{y\in \R} \int_{B(y,\kappa)} |u_n(x)|^2dx=0,
	\end{equation}
	for some $\kappa>0$, then $u_n \to 0$ strongly in $L^p(\R)$ for $q\in (2, \infty)$.
\end{lemma}


Next, we will recall and prove some technical results involving exponential critical growth.
\begin{lemma}\label{lm03}
\cite{Souza16} Let $\alpha >0$ and $r>1$. Then for each $\beta >r$ there exists $C=C(\beta)>0$ such that
\begin{equation*}\label{03}
\left(e^{\alpha |s|^2} - 1\right)^r \leq C\left(e^{\alpha \beta |s|^2}-1\right).
\end{equation*}
\end{lemma}

\begin{lemma}\label{lm04}
If $p>2$ and $u\in H^{1/2}(\R)$. Then there exists $C>0$, such that 
$$
\int_{\R}\left( e^{\pi |u|^2} - 1\right)|u|^pdx \leq C\|u\|_{1/2}^{p}.
$$
\end{lemma}
\begin{proof}
Consider $r>1$ close to $1$ such that 
$$
r \|u\|_{1/2}^{2}< 1\;\;\mbox{and}\;\;r'p\geq 2\;\;\mbox{with}\;\;r' = \frac{r}{r-1}
$$ 
Let $\beta >r$ close to $r$, such that $\beta \|u\|_{1/2}^2 < 1$. Then by H\"older inequality, (\ref{I03}), (\ref{Plm01}) and Lemma \ref{lm03}, 
$$
\begin{aligned}
\int_{\R}\left(e^{\pi |u|^2} - 1 \right)|u|^pdx &\leq \left(\int_{\R}\left( e^{\pi |u|^2} - 1\right)^{r} dx\right)^{1/r}\|u\|_{L^{r'p}(\R)}^{p}\\
&\leq C^{1/r} \left(\int_{\R} \left( e^{\pi \beta |u|^2} - 1\right)dx\right)^{1/r}\|u\|_{L^{r'p}(\R)}^{p}\\
&\leq C^{1/r}C_{r'p}^{p} \left(\int_{\R} \left(e^{\pi \beta \|u\|_{1/2}^{2} \left| \frac{u}{\|u\|_{1/2}}\right|^2} - 1\right)dx \right)^{1/r} \|u\|_{1/2}^{p}\\
&\leq \tilde{C}\|u\|_{1/2}^{p}.
\end{aligned}
$$ 
\end{proof}

Arguing as Alves \cite{CA16}, we can get the next two results  
\begin{lemma}\label{PLlm04}
Let $(u_n) \subset H^{1/2}(\R)$ be a sequence such that
\begin{equation}\label{05}
\limsup_{n\to \infty} \|u_n\|_{1/2}^{2} < 1.
\end{equation}
Then, there exists $t>1$ close to $1$ and $C>0$ such that 
\begin{equation}\label{06}
\int_{\R} \left( e^{\pi |u_n|^2 }-1\right)^tdx \leq C\;\;\mbox{for all}\;\;n\in \N.
\end{equation}
\end{lemma}

\begin{corollary}\label{PLcor05}
Let $(u_n) \subset H^{1/2}(\R)$ be a sequence satisfying (\ref{05}). If $u_n \rightharpoonup u$ in $H^{1/2}(\R)$ and $u_n(x) \to u(x)$ a.e. in $\R$, as $n\to \infty$, then,
\begin{equation}\label{07}
F(u_n(x)) \to F(u(x))\;\;\mbox{in}\;\;L^1(-T,T),
\end{equation} 
\begin{equation}\label{08}
f(u_n(x))u_n(x) \to f(u(x))u(x)\;\;\mbox{in}\;\;L^1(-T,T),
\end{equation} 
\begin{equation}\label{09}
\int_{-R}^{R}f(u_n(x))\varphi (x)dx \to \int_{-R}^{R}f(u(x))\varphi (x)dx,
\end{equation} 
as $n\to \infty$ for all $\varphi \in H^{1/2}(\R)$ and $T>0$. In particular, if $\varphi \in C_0^{\infty}(\R)$ we have 
\begin{equation}\label{09*}
\int_{\R}f(u_n(x))\varphi (x)dx \to \int_{\R}f(u(x))\varphi (x)dx.
\end{equation} 

\end{corollary}


Associated to problem $(P_\infty)$, we have the functional $I_\infty : H^{1/2}(\R) \to \R$ defined as 
\begin{equation}\label{10}
I_\infty (u) = \frac{1}{4\pi}\iint_{\R^{2}} \frac{|u(x) - u(y)|^2}{|x-y|^{2}}dy dx +\frac{1}{2} \int_{\R}|u|^2dx - \int_{\R}\tilde{Q}F(u)dx.
\end{equation}
It is standard to show that $I_\infty \in C^1(H^{1/2}(\R), \R)$ with  
\begin{equation}\label{11}
I'_\infty(u)v = \frac{1}{2\pi}\iint_{\R^2} \frac{[u(x) - u(y)][v(x) - v(y)]}{|x-y|^{2}}dy dx + \int_{\R}u(x)v(x)dx - \int_{\R} \tilde{Q}f(u(x))v(x)dx
\end{equation}
for all $u,v \in H^{1/2}(\R)$. 

We start our analysis recalling that $I_{\infty}$ satisfies the mountain pass geometry 
\begin{lemma}\label{Plm02}
The functional $I_{\infty}$ satisfies the following conditions:
\begin{enumerate}
\item[(i)] There exist $\beta, \delta>0$, such that $I_{\infty}(u) \geq \beta$ if $\|u\|_{1/2} = \delta$.
\item[(ii)] There exists $e\in H^{1/2}(\mathbb{R}^N)$ with $\|e\|_{1/2} > \delta$ such that $I_{\infty}(e) <0$.
\end{enumerate}
\end{lemma}

\noindent 
Let $\Gamma_{\infty} = \{\gamma \in C([0,1], H^{1/2}(\mathbb{R})):\;\;\gamma (0) = 0, I_{\infty}(\gamma (1))<0\}$, from Lemma \ref{Plm02}, the mountain pass level
$$
c_\infty = \inf_{\gamma \in \Gamma_{\infty}} \sup_{t\in [0,1]} I_{\infty}(\gamma (t)) \geq \beta >0,
$$
is well defined, and the equality below holds 
\begin{equation}\label{12}
c_{\infty}= \inf_{u\in \mathcal{N}_{\infty}} I_{\infty}(u)
\end{equation}
where 
$$
\mathcal{N}_{\infty} = \{u\in H^{1/2}(\mathbb{R})\setminus \{0\}:\;\;I'_{\infty}(u)u=0\}
$$ 
is the Nehari manifold associated to $(P_\infty)$. 

We wold like point out that the arguments explored in \cite[Theorem 1.5]{PFAQJT} still holds for $N=1$ and $\alpha=1/2$, hence the ground state solution $u_\infty \in H^{s}(\mathbb{R}^N)$ satisfies the estimate below  
\begin{equation} \label{decaimento}
0<\frac{C_1}{|x|^{2}} \leq u_\infty(x) \leq \frac{C_2}{|x|^{2}}, \quad \mbox{for all} \quad |x| \geq 1.  
\end{equation}
For the case where $f$ is a power, we cite the paper \cite[Proposition 3.1]{FLS}

We recall that by a ground state we understand by a function $u_\infty \in H^{1/2}(\mathbb{R})$ satisfying 
$$
I_\infty(u_\infty)=c_\infty \quad \mbox{and} \quad I'_\infty(u_\infty)=0.
$$

\section{Proof of Theorem \ref{main1}}
In this section, we are going to prove Theorem \ref{main1}. We start our analysis by state a version of a Lions type lemma that is crucial in our approach, whose the proof follows as in \cite[Proposition 3.1]{CACT2}.  

\begin{proposition}\label{GSresult03}
Let $(u_n) \subset H_{\tilde{\Omega}}^{1/2}$ be a bounded sequence such that 
\begin{equation}\label{GS10}
\lim_{n\to \infty}\sup_{y\in \R}\int_{\Lambda(y,\kappa)} |u_n|^2dx =0,
\end{equation}
for some $\kappa>0$ and $\Lambda(y,\kappa) = (y-\kappa,y+\kappa)\cap \R\setminus (a,b)$ with $\Lambda(y,\kappa)\neq \emptyset$. Then,
\begin{equation}\label{GS11}
\lim_{n\to \infty}\int_{\R\setminus (a,b)}|u_n|^{p}dx = 0\;\;\mbox{for all}\;\;p\in (2,\infty).
\end{equation}
\end{proposition}

\noindent

We start our analysis by considering the functional $I: H_{\tilde{\Omega}}^{1/2} \to \R$ associated to problem (P) which is defined as 
\begin{equation}\label{GS01}
I(u) = \frac{1}{4\pi}\iint_{\R^2 \setminus (a,b)^2} \frac{|u(x) - u(y)|^2}{|x-y|^2}dy dx + \frac{1}{2}\int_{\R \setminus (a,b)} |u|^2dx - \int_{\R}Q(x)F(u)dx.
\end{equation} 
The same idea explored in the proof of Lemma \ref{Plm02} works well to show that $I$ also satisfies
the geometry conditions of mountain pass theorem. Thus, applying the mountain pass theorem
without Palais-Smale condition found in \cite{MW}, it follows that there exists a $(PS)_{c_1}$ sequence $(u_n) \subset H_{\tilde{\Omega}}^{1/2}$ such that
\begin{equation}\label{GS02}
I(u_n) \to c_1 \quad \mbox{and}\quad I'(u_n) \to 0\;\;\mbox{as $n\to \infty$},
\end{equation}
where 
\begin{equation} \label{c1}
c_1 = \inf_{u\in H_{\tilde{\Omega}}^{1/2}\setminus \{0\}} \sup_{\sigma \geq 0}I(\sigma u)>0.
\end{equation}
The next result shows an important relation between the levels $c_1$ and $c_\infty$.
\begin{proposition}\label{GSresult02}
Assume that $(Q_1)$ holds. Then 
\begin{equation}\label{GS03}
0< c_1 < c_\infty.
\end{equation} 
\end{proposition}
\begin{proof}
Let $u_\infty$ be a ground state solution of $(P_\infty)$. Define $u_n(x) = u_\infty(x-n)$ and let $\alpha_n \in (0,\infty)$ with $\alpha_n\to 1$ as $n\to \infty$, such that 
$$
I(\alpha_nu_n) = \max_{t\geq 0} I(t u_n),\;\;\; \forall n \in \mathbb{N},
$$ 
and 
$$
\frac{1}{2\pi}\iint_{\R^2\setminus (a,b)^2} \frac{|u_n(x) - u_n(y)|^2}{|x-y|^{2}}dy dx + \int_{\R \setminus (a,b)} u_n^2(x)dx = \int_{\R\setminus (a,b)}Q(x)\frac{f(\alpha_nu_n)}{(\alpha_nu_n)}u_n^2dx.
$$
Then, by $(Q_1)$,  
\begin{equation}\label{GS04}
\begin{aligned}
c_1 &\leq \max_{t\geq 0} I(t u_n) = I(\alpha_nu_n)\\
& = I_\infty(\alpha_nu_n) - \frac{1}{4\pi}\iint_{(a,b)^2} \frac{|(\alpha_nu_n)(x) - (\alpha_nu_n)(y)|^2}{|x-y|^2}dy dx - \frac{1}{2}\int_{a}^{b}(\alpha_nu_n)^2dx \\
&+ \int_{\R \setminus (a,b)} [\tilde{Q} - Q(x)] F(\alpha_nu_n)dx + \int_{a}^{b} \tilde{Q}F(\alpha_nu_n)dx\\
&\leq I_\infty(\alpha_nu_n) - \frac{\alpha_n^2t_n}{2} + \int_{a}^b \tilde{Q}F(\alpha_nu_n)dx 
\end{aligned}
\end{equation}
where 
$$
t_n = \frac{1}{2\pi} \iint_{(a,b)^2} \frac{|u_n(x) - u_n(y)|^2}{|x-y|^{2}}dy dx + \int_{a}^{b}u_n^2(x)dx.
$$
Fixed $p>2$, by ($f_1$) and ($f_3$), for each $\tau >1$ and $\epsilon >0$ there exists $C_\epsilon>0$ such that
\begin{equation}\label{GS05}
|F(s)| \leq \epsilon |s|^2 + C_\epsilon |s|^p\left(e^{\pi \tau |s|^2} - 1 \right),\;\;\forall s\in \R,
\end{equation} 
and so,  
\begin{equation}\label{GS06}
\begin{aligned}
\int_{a}^{b}\tilde{Q}F(\alpha_nu_n)dx 
&\leq \epsilon \tilde{Q} \alpha_n^2 t_n + C_\epsilon \int_{a}^b \tilde{Q}\alpha_n^p|u_n|^p \left(e^{\pi \tau \alpha_n^2 |u_n|^2} -1 \right)dx.
\end{aligned}
\end{equation}
Combining (\ref{GS04}) with (\ref{GS06}), we obtain 
\begin{equation}\label{GS07}
c_1 \leq I_\infty (\alpha_nu_n) - t_n \left(\frac{\alpha_n^2}{2} - O(\epsilon) \right) + s_n,
\end{equation}
where 
$$
s_n = C_\epsilon  \tilde{Q} \alpha_n^p\int_{a}^{b}|u_n|^p\left(e^{\pi \tau \alpha_n^{2}|u_n|^2} - 1 \right)dx.  
$$
We claim that
\begin{equation}\label{GS08}
\lim_{n\to \infty} \frac{s_n}{t_n} = 0.
\end{equation}
In fact, first of all, note that $u_\infty\in H^{1/2}(\R)$ yields $t_n\to 0$ as $n\to +\infty$. On the other hand, from Lemma \ref{lm04}, there is a positive constant $C$, independent of $n$, such that 
\begin{equation}\label{GS09}
s_n \leq Ct_n^p.
\end{equation}  
Thus, 
$$
\frac{s_n}{t_n} \leq C t_n^{p-1} \to 0\;\;\mbox{as}\;\;n\to \infty, 
$$
proving the lemma.
\end{proof}

In what follows, we denote by $\xi$ the positive constant such that the extension operator $E: H_{\tilde{\Omega}}^{1/2} \to H^{1/2}(\R)$ satisfies 
\begin{equation}\label{GS17}
\|Eu\|_{1/2} \leq \xi \|u\|_{H_{\tilde{\Omega}}^{1/2}}\;\;\forall u\in H_{\tilde{\Omega}}^{1/2}.
\end{equation}  
For more details extension operator see \cite[Proposition 4.43]{FDGD}.

\begin{proposition}\label{GSresult04}
Let $(u_n) \subset H_{\tilde{\Omega}}^{1/2}$ be a sequence with $u_n \rightharpoonup 0$ and 
\begin{equation} \label{EQNOVA1}
\limsup_{n\to +\infty}\|u_n\|_{H_{\tilde{\Omega}}^{1/2}}^{2}\leq m < \frac{1}{2\xi^2}.
\end{equation}
If there is $\kappa>0$ such that 
\begin{equation}\label{GS18}
\lim_{n\to +\infty}\sup_{y\in \R}\int_{\Lambda (y,\kappa)}|u_n(x)|^2dx =0
\end{equation}
and $(f_1)-(f_5)$ hold, then 
\begin{equation}\label{GS19}
\lim_{n\to +\infty} \int_{\R\setminus (a,b)} F(u_n)dx = \lim_{n\to +\infty}\int_{\R\setminus (a,b)}f(u_n)u_ndx =0.
\end{equation}
\end{proposition}
\begin{proof}
By using (\ref{GS18}) together with Proposition \ref{GSresult03}, we get 
\begin{equation}\label{GS20}
u_n \to 0\;\;\mbox{in}\;\;L^{p}(\R\setminus (a,b))\;\;\mbox{for all $p\in (2,\infty)$}.
\end{equation} 
Setting $v_n = Eu_n$, it follows from (\ref{GS17})-(\ref{EQNOVA1}), 
$$
\|v_n\|_{1/2}\leq \xi \|u_n\|_{H_{\tilde{\Omega}}^{1/2}} < \frac{1}{\sqrt{2}} < 1, \quad \forall n \in \mathbb{N}.
$$  
Then, by (\ref{I03}) and Lemma \ref{lm03}, there exists $t>1$ close to $1$ such that 
$$
\sup_{n \in \mathbb{N}}\int_{\R}\left(e^{\pi |v_n|^2}-1 \right)^tdx<+\infty.
$$
From this, the function
$$
g_n(x) = e^{\pi |v_n(x)|^2} - 1,\;\;\forall x\in \R,
$$
belongs to $L^t(\R)$ and there exists $C>0$ such that 
$\|g_n\|_{L^t(\R)}\leq C$ for all $n\in \N$. Therefore, the sequence 
$$
h_n(x) = e^{\pi |u_n|^2} -1\;\;x\in \R\setminus (a,b)
$$ 
belongs to $L^t(\R\setminus (a,b))$ and there exists $C>0$ such that $\|h_n\|_{L^t(\tilde{\Omega})}\leq C$ for all $n\in \N$. On the other hand, by $(f_1)$ and $(f_3)$, given $\epsilon >0$, there is $C_\epsilon>0$ such that 
\begin{equation}\label{01}
|f(s)| \leq \epsilon |s| + C_\epsilon \left(e^{\pi |s|^2} - 1 \right),\;\forall s\in \R,
\end{equation}
and so, 
$$
|f(u_n)| \leq \epsilon |u_n| + C_\epsilon \left(e^{\pi |u_n|^2} - 1 \right), \forall n \in \mathbb{N}.
$$
Now, the H\"older inequality leads to
$$
\begin{aligned}
\int_{\R\setminus (a,b)} f(u_n)u_n 
&\leq \epsilon \int_{\R\setminus (a,b)}u_n^2dx + C_\epsilon C\|u_n\|_{L^{t'}(\tilde{\Omega})},
\end{aligned}
$$
with $\frac{1}{t}+\frac{1}{t'} = 1$. Recalling that $(u_n)$ is bounded in $H_{\tilde{\Omega}}^{1/2}$, the last inequality yields  
$$
\lim_{n\to \infty} \int_{\R\setminus (a,b)} f(u_n)u_ndx =0.
$$
The same idea works to show that 
$$
\lim_{n\to \infty}\int_{\R\setminus (a,b)}F(u_n)dx = 0.
$$
\end{proof}

\begin{proposition}\label{GSresult05}
If $(u_n) \subset H_{\tilde{\Omega}}^{1/2}$ satisfies 
$$
I(u_n) \to c_1\;\;\mbox{and}\;\;I'(u_n)\to 0,
$$
we have that 
\begin{equation} \label{NEQ1}
\limsup_{n\to +\infty}\|u_n\|_{H_{\tilde{\Omega}}^{1/2}} < \frac{1}{\sqrt{2}\xi}.
\end{equation}
Moreover, the weak limit $u_1$ of $(u_n)$ in $H_{\tilde{\Omega}}^{1/2}$ is a nontrivial critical point of $I$ with $I(u_1) = c_1$.
\end{proposition}
\begin{proof}
From ($f_1$)-($f_5$),  
\begin{equation}\label{GS21}
c_\infty < \frac{\theta-2}{4\xi^2\theta}.
\end{equation}
On the other hand, from (\ref{NEQ1}) and $(f_2)$, there exists $n_0\in \N$ such that for all $n\geq n_0$, it holds
$$
\begin{aligned}
\left(\frac{1}{2}-\frac{1}{\theta} \right)\|u_n\|_{H_{\tilde{\Omega}}^{1/2}}^{2} & \leq \left( \frac{1}{2}-\frac{1}{\theta}\right)\|u_n\|_{H_{\tilde{\Omega}}^{1/2}}^{2} + \int_{\R\setminus (a,b)}\left[\frac{1}{\theta}f(u_n)u_n - F(u_n) \right]dx\\
&= I(u_n) - \frac{1}{\theta}I'(u_n)u_n\\
&\leq c_1 + \|u_n\|_{H_{\tilde{\Omega}}^{1/2}}.
\end{aligned}
$$
Therefore $(u_n)$ is bounded in $H_{\tilde{\Omega}}^{1/2}$. Since $H_{\tilde{\Omega}}^{1/2}$ is a Hilbert space, up to a subsequence still denoted by $(u_n)$, there is $u_1 \in H_{\tilde{\Omega}}^{1/2}$ such that
\begin{equation*}\label{GS22}
\begin{aligned}
&u_n \rightharpoonup u_1\;\;\mbox{in}\;\;H_{\tilde{\Omega}}^{1/2},\\
&u_n \to u_1\;\;\mbox{in}\;\;L_{loc}^{q}(\R)\;\;\mbox{for all}\;\;q\geq 1,\\
&u_n(x) \to u_1(x)\;\;\mbox{a.e. in}\;\;\R.
\end{aligned}
\end{equation*}
By using again (\ref{NEQ1}) and ($f_2$),  we derive that
$$
\begin{aligned}
c_1 & = \lim_{n\to \infty} I(u_n)\\
& = \lim_{n\to \infty} \left(I (u_n) - \frac{1}{\theta}I'(u_n)u_n \right)\\
&\geq  \frac{\theta-2}{2\theta} \limsup_{n\to \infty}\|u_n\|_{H_{\tilde{\Omega}}^{1/2}}^{2}.
\end{aligned}
$$  
Hence, by Proposition \ref{GSresult02} and (\ref{GS21}), 
\begin{equation}\label{GS23}
\limsup_{n\to \infty}\|u_n\|_{H_{\tilde{\Omega}}^{1}}^{2} = m \leq \frac{2\theta c_1}{\theta - 2} < \frac{1}{2\xi^2}.
\end{equation}
Then, there exists $n_0\in \N$ such that 
\begin{equation}\label{GS24}
\|u_n\|_{H_{\tilde{\Omega}}^{1/2}}< \frac{1}{\sqrt{2}\xi}\;\;\forall n\geq n_0.
\end{equation}
Consequently, for $v_n = Eu_n$ we get 
$$
\|v_n\|_{1/2} \leq \xi \|u_n\|_{H_{\tilde{\Omega}}^{1/2}}
$$ 
and by (\ref{GS24}) 
\begin{equation}\label{GS25}
\|v_n\|_{1/2}\leq \xi\|u_n\|_{H_{\tilde{\Omega}}^{1/2}} < \frac{1}{\sqrt{2}} < 1,\;\;\forall n\geq n_0.
\end{equation}
By (\ref{I03}), there exist $\beta, t>1$ close to $1$ with $\beta \|v_n\|_{1/2}^{2}< 1$ such that 
$$
\int_{\R}\left(e^{\pi |v_n|^2}-1 \right)^tdx \leq \int_{\R}\left(e^{\pi \beta \|v_n\|_{1/2}^{2} \left|\frac{v_n}{\|v_n\|_{1/2}} \right|^2}-1 \right)dx \leq C.
$$
Thus,  the function 
$$
f_n(x) = e^{\pi |v_n(x)|^2} - 1\;\;\forall x\in \R,
$$
belongs to $L^t(\R)$ and there exists $C>0$ such that $\|f_n\|_{L^t(\R)}\leq C$ for all $n\in \N$. Hence, the sequence 
\begin{equation}\label{GS26}
h_n(x) = e^{\pi |u_n(x)|^2} -1\;\;x\in \R\setminus (a,b)
\end{equation}
belongs to $L^t(\R\setminus (a,b))$ and there exists $C>0$ such that $\|h_n\|_{L^t(\tilde{\Omega})}\leq C$ for all $n\in \N$. 

Let $\varphi\in H^{1/2}(\tilde{\Omega})$ a test function with bounded support, then  $I'(u_n)\varphi = o_n(1),$
that is, 
$$
\frac{1}{2\pi}\iint_{\R^2\setminus (a,b)^2}\frac{[u_n(x) - u_n(y)][\varphi (x) - \varphi (y)]}{|x-y|^2}dy dx + \int_{\R \setminus (a,b)} u_n(x)\varphi (x)dx = \int_{\R\setminus (a,b)} Q(x)f(u_n(x))\varphi (x)dx.
$$
By the weak convergence $u_n \rightharpoonup u$ in $H_{\tilde{\Omega}}^{1/2}$, 
\begin{equation}\label{GS27}
\begin{aligned}
\frac{1}{2\pi}\iint_{\R^2\setminus (a,b)^2}&\frac{[u_n(x) - u_n(y)][\varphi (x) - \varphi (y)]}{|x-y|^2}dy dx + \int_{\R \setminus (a,b)} u_n(x)\varphi (x)dx\\
&\to \frac{1}{2\pi}\iint_{\R^2\setminus (a,b)^2}\frac{[u_1(x) - u_1(y)][\varphi (x) - \varphi (y)]}{|x-y|^2}dy dx + \int_{\R \setminus (a,b)} u_1(x)\varphi (x)dx
\end{aligned}
\end{equation}
as $n\to \infty$. On the other hand 
$$
\begin{aligned}
\int_{\R\setminus (a,b)}Q(x)[f(u_n(x)) - f(u_1(x))]\varphi(x)dx & = \int_{B(0,T)\cap {(\R\setminus (a,b))}} Q(x)[f(u_n(x)) - f(u_1(x))]\varphi (x) dx \\
&+\int_{B^c(0,T)\cap (\R\setminus (a,b))}Q(x)[f(u_n(x)) - f(u_1(x))]\varphi (x)dx\\
& = A_1 + A_2.
\end{aligned}
$$
For $A_2$, the boundedness of $Q$ combined with H\"older inequality gives  
$$
\begin{aligned}
\int_{B^c(0,T)\cap(\R\setminus (a,b))}\hspace{-2cm}Q(x)[f(u_n) - f(u_1)]\varphi (x)dx &\leq K\left( \int_{B^c(0,T)\cap(\R\setminus (a,b))} |f(u_n) - f(u_1)|^tdx\right)^{1/t} \left( \int_{B^c(0,T)\cap (\R\setminus (a,b))} |\varphi|^{t'}dx\right)^{1/t'}\\
&\leq K\left( \int_{\R\setminus (a,b)} |f(u_n) - f(u_1)|^tdx\right)^{1/t} \left( \int_{B^c(0,T)} |\varphi|^{t'}dx\right)^{1/t'},
\end{aligned}
$$
where $t>1$ is close to $1$ and 
$$
\frac{1}{t} + \frac{1}{t'} = 1.
$$
By (\ref{GS26}), there exist a positive constant $C$ such that 
$$
\int_{B^c(0,T)\cap(\R\setminus (a,b))}Q(x)[f(u_n) - f(u_1)]\varphi (x)dx \leq \tilde{K}\left( \int_{B^c(0,T)} |\varphi|^{t'}dx\right)^{1/t'}.
$$
So for any $\epsilon >0$, there exists $T_0>0$ such that 
\begin{equation}\label{GS28}
\int_{B^c(0,T)\cap(\R\setminus (a,b))}Q(x)[f(u_n) - f(u_1)]\varphi (x)dx < \frac{\epsilon}{2}\;\;\mbox{for all}\;\;T>T_0.
\end{equation}
On the other hand, following the ideas of Corollary \ref{PLcor05} we can show that 
\begin{equation}\label{GS29}
 \int_{B(0,T)\cap {(\R\setminus (a,b))}} Q(x)[f(u_n(x)) - f(u_1(x))]\varphi (x) dx \to 0\;\;\mbox{as}\;\;n\to \infty.
\end{equation}  
Therefore, taking the limit in $I'(u_n)\varphi=o_n(1)$ and using (\ref{GS27})-(\ref{GS29}), we obtain  
$$
\frac{1}{2\pi}\iint_{\R^2\setminus (a,b)^2}\frac{[u_1(x)-u_1(y)][\varphi(x) - \varphi(y)]}{|x-y|^2}dy dx + \int_{\R\setminus (a,b)}u_1(x)\varphi(x)dx  = \int_{\R\setminus (a,b)}Q(x)f(u_1(x))\varphi(x)dx,
$$ 
for all $\varphi\in H_{\tilde{\Omega}}^{1/2}(\R)$, that is,  
$$
I'(u_1)\varphi = 0\;\;\mbox{for all}\;\;\varphi\in H_{\tilde{\Omega}}^{1/2}.
$$ 

\noindent 
Now, we are going to show that $u\neq 0$. Assuming by contradiction that $u=0$, we have two situations to consider:
\begin{enumerate}
\item[(I)] $\displaystyle \lim_{n\to \infty}\sup_{y\in \R} \int_{\Lambda(y,\kappa)}|u_n|^2dx=0$, or
\item[(II)] There exist $\eta>0$ and $(y_n) \subset \R$ with $|y_n| \to +\infty$ such that 
$$
\liminf_{n\to \infty}\int_{\Lambda(y_n, \kappa)} |u_n|^2dx\geq \eta.
$$
\end{enumerate}
In the sequel, we prove that the aforementioned cases (I) and (II) do not hold, thus we can conclude that $u\neq 0$.

\noindent 
{\bf Analysis of (I):} If (I) holds, it follows from Proposition \ref{GSresult04},  
$$
\lim_{n\to \infty} \int_{\R\setminus (a,b)} f(u_n)u_ndx =0.
$$
Combining this equality with 
$$
o(1) = I'(u_n)u_n = \|u_n\|_{H_{\tilde{\Omega}}^{1/2}}^{2}
- \int_{\R\setminus (a,b)}Q(x)f(u_n)u_ndx, 
$$
we conclude that $\|u_n\|_{H_{\tilde{\Omega}}^{1/2}}\to 0$ as $n\to \infty$, which is absurd, because $I(u_n)\to c_1>0$. Therefore, (I) does not hold.  

\noindent 
{\bf Analysis of (II):} Let $w_n(x) = u_n(x+y_n)$ for $x\in \R\setminus (a-y_n,b-y_n)$. Since $|y_n| \to \infty$ as $n\to \infty$, for each $T>0$ fixed, there is $n_0=n_0(T) \in \mathbb{N}$ such that  
$$
(-T,T) \subset \R\setminus (a-y_n,b-y_n), \quad \forall n \geq n_0.
$$ 
In what follows we set $\tilde{\Omega}_n = \R\setminus (a-y_n, b-y_n).$ As $(u_n)$ is bounded in $H_{\tilde{\Omega}}^{1/2}$, then there exists a positive constant $M$ such that 
$$
\begin{aligned}
M & \geq  \frac{1}{2\pi}\iint_{\R^{2} \setminus (a,b)^2} \frac{|u_n(x) - u_n(y)|^2}{|x-y|^{2}}dy dx + \int_{\R\setminus (a,b)}|u_n|^2dx\\
&= \frac{1}{2\pi}\iint_{\R^{2} \setminus (a-y_n,b-y_n)^2} \frac{|w_n(x) - w_n(y)|^2}{|x-y|^{2}}dy dx + \int_{\R \setminus (a-y_n,b -y_n)}|w_n|^2dx \\
&\geq \frac{1}{2\pi}\int_{-T}^{T}\int_{-T}^{T} \frac{|w_n(x) - w_n(y)|^2}{|x-y|^{2}}dy dx + \int_{-T}^{T} |w_n|^2dx,
\end{aligned}
$$  
that is, 
$$
\|w_n\|_{H^{1/2}(-T,T)}^{2} \leq M, \quad \forall n \geq n_0 \quad \mbox{and} \quad \forall T>0. 
$$ 
From this, there is a subsequence of $(u_n)$, still denoted by itself, and $w\in H_{loc}^{1/2}(\R) \setminus \{0\}$ such that for each $T>0$,
$$
w_n \rightharpoonup  w \;\;\mbox{in $H^{1/2}(-T,T)$, as $n\to \infty$}.
$$
Then, by the lower semicontinuity of the norm  
$$
\|w\|_{H^{1/2}(-T,T)}\leq \liminf_{n\to \infty}\|w_n\|_{H^{1/2}(-T,T)} \leq M,\quad \forall T>0,
$$ 
from where it follows that $w \in H^{1/2}(\mathbb{R})$ and 
$$
\|w\|_{H^{1/2}(\R)}\leq \liminf_{T\to \infty} \|w\|_{H^{1/2}(-T,T)}\leq M. 
$$
Denoting $\hat{w}_n = Ew_n$, it follows that 
$$
\|\hat{w}_n\|_{1/2} \leq \xi \|w_n\|_{H_{\tilde{\Omega}_n}^{1/2}},
$$
then
\begin{equation}\label{GS30}
\|\hat{w}_n\|_{1/2} \leq \xi \|u_n\|_{H_{\tilde{\Omega}}^{1/2}} < \frac{1}{\sqrt{2}},\;\;\mbox{for all $n\in \N$}.
\end{equation}
Now, let $\psi \in H_{\tilde{\Omega}}^{1/2}$ be a test function with bounded support. Since $I'(u_{n}) = 0$, we have    
\begin{equation}\label{GS31}
I'(u_n)\psi(.-y_{n}) = 0.
\end{equation}
By doing the change of variable $\tilde{x} = x - y_{n}$ and $\tilde{y} = y - y_{n}$, we get 
\begin{equation}\label{GS32}
\begin{aligned}
&\frac{1}{2\pi}\iint_{\R^{2}\setminus (a-y_n,b-y_n)^2} \frac{[w_{n}(x) - w_{n}(y)][\psi(x) - \psi(y)]}{|x-y|^{2}}dy dx + \int_{\tilde{\Omega}_n} w_{n}(x)\psi(x)dx \\
&\hspace{6cm}= \int_{\tilde{\Omega}_n} Q(x + y_{n})f(w_n(x))\psi(x)dx.
\end{aligned}
\end{equation} 
By the weak convergence of $w_n$ to $w$,  
\begin{equation}\label{GS33}
\begin{aligned}
\frac{1}{2\pi}\iint_{\R^{2}\setminus (a-y_n, b-y_n)^2} &\frac{[w_{n}(x) - w_{n}(y)][\psi(x) - \psi(y)]}{|x-y|^{2}}dy dx + \int_{\tilde{\Omega}_n} w_{n}(x)\psi(x)dx \\
&\to \frac{1}{2\pi} \iint_{\R^{2}} \frac{[w(x) - w(y)][\psi(x) - \psi(y)]}{|x-y|^{2}}dydx + \int_{\R} w(x)\psi(x)dx.
\end{aligned}
\end{equation}
Now we are going to show that 
$$
\int_{\tilde{\Omega}_n} Q(x+y_n)f(w_n(x))\psi(x)dx \to \int_{\R} \tilde{Q}f(w(x))\psi(x)dx.
$$
In fact
$$
\begin{aligned}
\int_{\tilde{\Omega}_n} Q(x+y_n)f(w_n)\psi dx - \int_{\R}\tilde{Q}f(w)\psi dx &= \int_{\tilde{\Omega}_n} [Q(x+y_n)f(w_n) - \tilde{Q}f(w)]\psi dx\\
&-\int_{a-y_n}^{b-y_n}\tilde{Q}f(w)\psi dx
\end{aligned}
$$
By H\"older inequality,  
$$
\begin{aligned}
\int_{a-y_n}^{b-y_n}\tilde{Q}f(w_n)\psi dx &= \tilde{Q}\int_{\R} \chi_{(a-y_n,b-y_n)}(x)f(w_n)\psi dx\\
&\leq \tilde{Q} \left(\int_{\R} |f(w_n)|^tdx\right)^{1/t}\left(\int_{\R}\chi_{(a-y_n.b-y_n)}(x) |\psi|^{t'}dx\right)^{1/t'},
\end{aligned}
$$
where $t>1$ close to 1 is such that $f(w_n)\in L^t(\R)$ and there exists $C>0$ such that $\|f(w_n)\|_{L^t(\R)}\leq C$ and 
$$
\frac{1}{t} + \frac{1}{t'} = 1.
$$ 
Note that,
$$
\chi_{(a-y_n, b-y_n)}(x) \to 0\;\;\mbox{as $n\to \infty$, for a.e. $x\in \R$}.
$$
then 
$$
\chi_{(a-y_n, b-y_n)}(x)|\psi(x)|^{t'} \to 0\;\;\mbox{as $n\to +\infty$, for a.e. $x\in \R$}.
$$
Furthermore 
$$
|\chi_{(a-y_n, b-y_n)}(x) |\psi|^{t'}| \leq |\psi|^{t'} \in L^1(\R).
$$
Thereby, by Lebesgue's theorem,  
$$
\int_{\R} \chi_{a-y_n,b-y_n}(x)|\psi|^{t'} dx \to 0\;\;\mbox{as}\;\;n\to +\infty, 
$$
which implies 
\begin{equation}\label{GS34}
\int_{a-y_n}^{b-y_n}\tilde{Q}f(w_n)\psi dx \to 0,\;\,\mbox{as}\;\;n\to \infty.
\end{equation}

From now on, we fix $T>0$ and $n_0 \in \mathbb{N}$ such that  
$$
(-T,T)\subset \R\setminus (a-y_n,b-y_n) \;\;\mbox{for all $n\geq n_0$}.
$$  
Then 
$$
\begin{aligned}
\int_{\tilde{\Omega}_n} [Q(x+y_n)f(w_n) - \tilde{Q}f(w)]\psi dx &= \int_{(-T,T) \cap \tilde{\Omega}_n} [Q(x+y_n)f(w_n) - \tilde{Q}f(w)]\psi dx \\
&+ \int_{(\R\setminus (-T,T)) \cap \tilde{\Omega}_n} [Q(x+y_n)f(w_n) - \tilde{Q}f(w)]\psi dx.
\end{aligned}
$$
Setting $\Lambda_n = (\R\setminus (-T,T)) \cap \tilde{\Omega}_n$, it follows from H\"older inequality 
$$
\begin{aligned}
\int_{\Lambda_n} [Q(x+y_n)f(w_n) - \tilde{Q}f(w)]\psi dx &= \int_{\Lambda_n} (Q(x+y_n) - \tilde{Q})f(w_n)\psi dx +\int_{\Lambda_n} \tilde{Q}(f(w_n) - f(w))\psi dx \\
&\leq \left(\int_{\Lambda_n}|f(w_n)|^tdx\right)^{1/t}  \left(\int_{\Lambda_n} |[Q(x+y_n) - \tilde{Q}]\psi|^{t'} dx\right)^{1/t'} \\
&+ \tilde{Q} \left(\int_{\Lambda_n}|f(w_n) - f(w)|^{t}dx\right)^{1/t} \left(\int_{\Lambda_n}|\psi|^{t'} dx\right)^{1/t'}\\
&\leq \left(\int_{\tilde{\Omega}_n}|f(w_n)|^tdx\right)^{1/t}  \left(\int_{\R\setminus (-T,T)} |[Q(x+y_n) - \tilde{Q}]\psi|^{t'} dx\right)^{1/t'} \\
&+\tilde{Q} \left(\int_{\tilde{\Omega}_n}|f(w_n) - f(w)|^{t}dx\right)^{1/t} \left(\int_{\R\setminus (-T,T)}|\psi|^{t'} dx\right)^{1/t'}.
\end{aligned}
$$
Note that, for a.e. $x\in \R\setminus (-T,T)$ 
$$
|Q(x+y_n) - \tilde{Q}|^{t'}|\psi(x)|^{t'} \to 0 \;\;\mbox{as}\;\;n\to \infty,
$$
and since $Q$ is bounded, we have 
$$
|Q(x+y_n) - \tilde{Q}|^{t'}|\psi|^{t'}\leq K|\psi|^{t'}\in L^1(\R).
$$ 
Then, by Lebesgue's Theorem 
$$
\int_{\R\setminus (-T,T)} |[Q(x+y_n) - \tilde{Q}]\psi|^{t'} dx \to 0\;\;\mbox{as}\;\,n\to 0.
$$
Observing that $(f(\hat{w}_n))$ is bounded in $L^{t}(\R)$ and $f(w) \in L^{t}(\R)$, given $\epsilon>0$, we can take  $T>0$ large enough such that 
$$
\tilde{Q} \left(\int_{\tilde{\Omega}_n}|f(w_n) - f(w)|^{t}dx\right)^{1/t} \left(\int_{\R\setminus (-T,T)}|\psi|^{t'} dx\right)^{1/t'} < \epsilon, \quad \forall n \in \mathbb{N}.
$$
The above analysis give
\begin{equation}\label{GS35}
\int_{\Lambda_n} [Q(x+y_n)f(w_n) - \tilde{Q}f(w)]\psi dx \to 0.
\end{equation}
On the other hand, as in Corollary \ref{PLcor05}, we have 
\begin{equation}\label{GS36}
\int_{(-T,T) \cap \tilde{\Omega}_n} [Q(x+y_n)f(w_n) - \tilde{Q}f(w)]\psi dx = \int_{(-R,R)} [Q(x+y_n)f(w_n) - \tilde{Q}f(w)]\psi dx \to 0\;\,\mbox{as}\;\;n\to +\infty.
\end{equation}
Finally, by (\ref{GS33})-(\ref{GS36}), we get 
$$
I'_\infty(w)\psi = 0,
$$
showing that $w$ is a nontrivial weak solution of $(P_\infty)$. Now, applying the Fatou's lemma we get the inequality below
$$
\begin{aligned}
c_\infty &\leq I_\infty (w) - \frac{1}{2}I'_\infty(w)w \leq \liminf_{n\to \infty}I(u_n) = c_1 < c_\infty,
\end{aligned}
$$
which contradicts Proposition \ref{GSresult02}.
\end{proof}

\noindent
{\bf Proof of Theorem \ref{main1}.} First of all, in order to find positive ground state solution we assume that 
$$
f(s)=0\;\;\mbox{for all $s\leq 0$}.
$$
By Proposition \ref{GSresult05} and the Mountain Pass Theorem, $I$ has a critical point $u_1$ at the level set $c_1$.


\section{Nodal Solution}

\noindent
We star this section recalling an important lemma that will be used later on, whose the proof can be found in \cite{CAPCEM}
\begin{lemma}\label{alves}
	Let $F\in C^2(\R, \R^+)$ be a convex and even function such that $F(0) = 0$ and $f(s)=F'(s)\geq 0$ for all $s\in [0,\infty)$. Then, for all $u,v\geq 0$
	$$
	|F(u-v) - F(u) - F(v)|\leq 2(f(u)v + f(v)u).
	$$
\end{lemma}

In what follows, we introduce the nodal set
$$
\mathcal{M} = \{u\in \mathcal{N}:\;\;u^{\pm} \not \equiv 0, I'(u)u^+ = I'(u)u^-=0\}
$$
and consider the following real number 
$$
c = \inf_{u\in \mathcal{M}} I(u).
$$
Let us point out that for all $u\in H_{\tilde{\Omega}}^{1/2}$,
\begin{equation}\label{N01}
[ u ]_{H_{\tilde{\Omega}}^{1/2}}^{2} = [u^+]_{H_{\tilde{\Omega}}^{1/2}}^{2} + [u^-]_{H_{\tilde{\Omega}}^{1/2}}^{2}- \frac{1}{2\pi}\iint_{\R^{2}\setminus (a,b)^2}\frac{u^+(x)u^-(y) + u^-(x)u^+(y)}{|x-y|^{2}}dy dx,
\end{equation}
where
$$
[u]_{H_{\tilde{\Omega}}^{1/2}}^{2} = \frac{1}{2\pi}\iint_{\R^{2}\setminus (a,b)^2}\frac{|u(x) - u(y)|^2}{|x-y|^{2}}dy dx.
$$
Therefore
\begin{equation}\label{N02}
I(u) = I(u^+) + I(u^-) - \frac{1}{2\pi}\iint_{\R^{2}\setminus (a,b)^2}\frac{u^+(x)u^-(y) + u^-(x)u^+(y)}{|x-y|^{2}}dy dx,
\end{equation}
\begin{equation}\label{N03}
I'(u)u^+ = I'(u^+)u^+ - \frac{1}{\pi}\iint_{\R^{2}\setminus (a,b)^2}\frac{u^+(x)u^-(y) + u^-(x)u^+(y)}{|x-y|^{2}}dy dx
\end{equation}
and 
\begin{equation}\label{N04}
I'(u)u^- = I'(u^-)u^- - \frac{1}{\pi}\iint_{\R^{2}\setminus (a,b)^2}\frac{u^+(x)u^-(y) + u^-(x)u^+(y)}{|x-y|^{2}}dy dx.
\end{equation}
In particular, since for any $u \in H_{\tilde{\Omega}}^{1/2}$
$$
\frac{1}{\pi}\iint_{\R^{2N}\setminus \Omega^2}\frac{u^+(x)u^-(y) + u^-(x)u^+(y)}{|x-y|^{2}}dy dx \leq 0,
$$
then for any $u\in \mathcal{M}$
\begin{equation}\label{N05}
I'(u^{\pm})u^{\pm} \leq 0.
\end{equation}
\begin{lemma}\label{Nlm01}
There exists $\rho >0$ such that 
\begin{enumerate}
\item[$(i)$] $\|u\|_{H_{\tilde{\Omega}}^{1/2}} \geq \rho$ and $I(u) >0 $ for all $u\in \mathcal{N}$;
\item[$(ii)$] $\|u^{\pm}\|_{H_{\tilde{\Omega}}^{1/2}} \geq \rho$ for all $u \in \mathcal{M}$.
\end{enumerate}
\end{lemma}
\begin{proof}
\begin{enumerate}
\item[$(i)$] By ($f_3$), given $p>2$ and $\epsilon >0$, there exists $C_\epsilon >0$ such that 
\begin{equation}\label{N06}
|f(s)| \leq \epsilon |s| + C_\epsilon |s|^{p-1}\left(e^{\pi s^2} - 1 \right),
\end{equation}
By contradiction, suppose that there exists a sequence $(u_n)\subset \mathcal{N}$ such that 
\begin{equation}\label{N07}
\|u_n\|_{H_{\tilde{\Omega}}^{1/2}}\to 0\;\;\mbox{as}\;\;n\to \infty.
\end{equation} 
Since $u_n\in \mathcal{N}$, (\ref{N06}) together with H\"older inequality yields 
$$
\begin{aligned}
\|u_n\|_{H_{\tilde{\Omega}}^{1/2}}^{2} &=\int_{\R\setminus (a,b)}Q(x)f(u_n)u_ndx\\
&\leq \|Q\|_{\infty}S_2^2 \epsilon \|u_n\|_{H_{\tilde{\Omega}}^{1/2}}^{2} + \|Q\|_{\infty}S_{2p}^{p}C_\epsilon \|u_n\|_{H_{\tilde{\Omega}}^{1/2}}^{p} \left( \int_{\R\setminus (a,b)} \left( e^{2\pi u_n^2} - 1\right)dx\right)^{1/2},
\end{aligned}
$$
or equivalently
$$
(1-\|Q\|_{\infty}S_{2}^{2}\epsilon)\|u_n\|_{H_{\tilde{\Omega}}^{1/2}}^{2} \leq \|Q\|_{\infty}S_{2p}^{p}C_\epsilon \|u_n\|_{H_{\tilde{\Omega}}^{1/2}}^{p} \left( \int_{\R\setminus (a,b)} \left( e^{2\pi u_n^2} - 1\right)dx\right)^{1/2}.
$$
Taking $\epsilon>0$ small enough such that 
$$
\tilde{C} = \frac{1-\|Q\|_{\infty}S_{2}^{2}\epsilon}{\|Q\|_{\infty}S_{2p}^{p}C_\epsilon} >0,
$$ 
we get
\begin{equation}\label{N08}
0< \tilde{C} \leq \|u_n\|_{H_{\tilde{\Omega}}^{1/2}}^{p-2} \left( \int_{\R\setminus (a,b)} \left( e^{2\pi u_n^2} - 1\right)dx\right)^{1/2}. 
\end{equation}
Setting $v_n = Eu_n$, we get  
$$
\|v_n\|_{1/2} \leq \xi \|u_n\|_{H_{\tilde{\Omega}}^{1/2}}.
$$
Since $\|u_n\|_{H_{\tilde{\Omega}}^{1/2}}\to 0$ as $n\to \infty$, there exists $n_0\in \N$ such that 
$$
2 \|v_n\|_{1/2}^{2} < 1,\;\;\forall n\geq n_0.
$$
Then by (\ref{I03}), 
$$
\int_{\R}\left(e^{2\pi |v_n|^2} - 1 \right)dx = \int_{\R} \left(e^{2\pi \|v_n\|_{1/2}^{2}\left|\frac{v_n}{\|v_n\|_{1/2}} \right|^2}-1 \right) \leq C\;\;\mbox{for every}\;\,n\geq n_0, 
$$
and so, 
$$
\int_{\R\setminus (a,b)}\left( e^{2\pi |u_n|^2} - 1\right)dx \leq C, \quad \forall n\geq n_0.
$$ 
Consequently, by (\ref{N08}),
$$
0 < \left( \frac{\tilde{C}}{\sqrt{C}}\right)^{\frac{1}{p-2}} \leq \|u_n\|_{H_{\tilde{\Omega}}^{1/2}},\quad \forall n\geq n_0,
$$
which is a contradiction. 
\end{enumerate}
\end{proof}

\begin{lemma}\label{Nlm02}
Suppose that $(Q_1)$ and $(Q_2)$ hold. Then 
\begin{equation}\label{N09}
c < c_1 + c_\infty,
\end{equation}
for $R$ given in $(Q_2)$ large enough. 
\end{lemma}
\begin{proof}
Let $u_\infty$, $u_1$ be a positive ground state solution solution of $(P_\infty)$ and $(P)$ respectively.   For $\alpha, \tau >0$ and $R>0$, let $\sigma=\sigma_R$ given in $(Q_2)$, $u_\sigma(x) = u_\infty(x-\sigma)$, 
$w_\sigma(x) = \alpha u_1(x) - \tau u_\sigma(x),$ and 
\begin{equation}\label{N10}
h_{\sigma}^{\pm}(\alpha, \tau) = I'(w_\sigma)w_{\sigma}^{\pm}.
\end{equation}
Recalling that  $I'(u_1)u_1 = 0$, ($f_4$) leads to
$$
\begin{aligned}
I'(\frac{u_1}{2})\frac{u_1}{2} 
& = \frac{1}{4}\int_{\R\setminus (a,b)} Q(x)f(u_1)u_1 dx - \int_{\R\setminus (a,b)} Q(x)f(\frac{u_1}{2})\frac{u_1}{2}dx\\
& = \int_{\R\setminus (a,b)}Q(x)\left(\frac{f(u_1)}{u_1} - \frac{f(u_1/2}{u_1/2} \right)\left(\frac{u_1}{2} \right)^2dx>0;\\
I'(2u_1)2u_1 
& = 4\int_{\R\setminus (a,b)}Q(x)f(u_1)u_1dx - \int_{\R\setminus (a,b)}Q(x)f(2u_1)2u_1dx\\
& = \int_{\R\setminus (a,b)} Q(x)\left(\frac{f(u_1)}{u_1} - \frac{f(2u_1)}{2u_1} \right)(2u_1)^2dx <0.
\end{aligned}
$$ 
We claim that there exists $\sigma_0>0$ such  
\begin{equation}\label{N11}
I'(\frac{u_\sigma}{2})\frac{u_\sigma}{2} >0,\;\;\mbox{for all}\;\;|\sigma|> \sigma_0.
\end{equation}
In fact, note that 
$$
\begin{aligned}
I'(\frac{u_\sigma}{2})\frac{u_\sigma}{2} 
&= \frac{1}{2}\left(\frac{1}{2\pi}\iint_{\R^2} \frac{|(u_\sigma/2)(x) - (u_\sigma/2)(y)|^2}{|x-y|^{2}}dy dx + \int_{\R} \left( \frac{u_\sigma}{2}\right)^2dx \right) - \int_{\R} Q(x)f(\frac{u_\sigma}{2})\frac{u_\sigma}{2}dx\\
& - \left(\frac{1}{2}\left(\frac{1}{2\pi}\iint_{(a,b)^2} \frac{|(u_\sigma/2)(x) - (u_\sigma/2)(y)|^2}{|x-y|^{2}}dy dx + \int_{(a,b)} \left( \frac{u_\sigma}{2}\right)^2dx \right) - \int_{(a,b)} Q(x)f(\frac{u_\sigma}{2})\frac{u_\sigma}{2}dx \right)\\
& = I_1 - I_2.
\end{aligned}
$$
Making the change of variable $\tilde{x} = x-\sigma$ and $\tilde{y} = y-\sigma$, we get 
$$
I_1 = \frac{1}{4\pi}\iint_{\R^2} \frac{|(u_\infty/2)(x) - (u_\infty/2)(y)|^2}{|x-y|^2} + \frac{1}{2}\int_{\R}\left(\frac{u_{\infty}}{2}\right)^{2}dx - \int_{\R}Q(x+\sigma)f(\frac{u_\infty}{2})\frac{u_\infty}{2}dx.
$$
Note that, for a.e. $x\in \R$,
$$
Q(x+\sigma)f(\frac{u_\infty(x)}{2})\frac{u_\infty(x)}{2} \to \tilde{Q}f(\frac{u_\infty(x)}{2})\frac{u_\infty(x)}{2}\;\,\mbox{as}\;\,|\sigma| \to \infty
$$ 
and
$$
Q(x+\sigma)f(\frac{u_\infty(x)}{2})\frac{u_\infty(x)}{2} \leq \|Q\|_{\infty} f(\frac{u_\infty(x)}{2})\frac{u_\infty(x)}{2} \in L^1(\R).
$$  
Since $f(\frac{u_\infty}{2})\frac{u_\infty}{2} \in L^1(\R)$, the Lebesgue's Theorem gives 
$$
\int_{\R}Q(x+\sigma)f(\frac{u_\infty(x)}{2})\frac{u_\infty(x)}{2}dx \to \int_{\R}\tilde{Q}f(\frac{u_\infty(x)}{2})\frac{u_\infty(x)}{2}dx\;\;\mbox{as}\;\;|\sigma| \to \infty,
$$ 
and so, 
\begin{equation}\label{N12} 
I_1 \to I'(\frac{u_\infty}{2})\frac{u_\infty}{2} >0\;\;\mbox{as}\;\;|\sigma|\to \infty.
 \end{equation}
Since  $u_\infty \in H^{1/2}(\R)$ and
 $$
\chi_{(a-\sigma, b-\sigma)}(x) \to 0\;\;\mbox{as $|\sigma| \to \infty$, for a.e. $x\in \R$},
$$
a similar argument works to prove that 
\begin{equation}\label{N13}
I_2 \to 0\;\,\mbox{as}\;\,|\sigma| \to \infty.
\end{equation}
Now (\ref{N11}) follows from (\ref{N12}) and (\ref{N13}). 

Now we claim that there exists $\sigma_0>0$ large enough, such that 
\begin{equation}\label{N14}
\begin{cases}
h_{\sigma}^{+}(\frac{1}{2}, \tau) >0\\
h_{\sigma}^{+}(2, \tau)<0
\end{cases}
\end{equation}  
for $|\sigma|> \sigma_0$ and $\tau \in [\frac{1}{2}, 2]$. A similar argument can be used to show that for  $\alpha \in [\frac{1}{2}, 2]$ the inequality below holds  
\begin{equation}\label{N15}
\begin{cases}
h_{\sigma}^{-}(\alpha, \frac{1}{2})>0\\
h_{\sigma}^{-}(\alpha, 2)< 0.
\end{cases}
\end{equation} 
In fact, note that
$$
\begin{aligned}
h_{\sigma}^{+}(\frac{1}{2}, \tau)&= I'(w_\sigma)w_{\sigma}^{+}\\
& = \frac{1}{2\pi}\iint_{\R^2\setminus (a,b)^2}\frac{[\frac{u_1}{2}(x) - \frac{u_1}{2}(y)][w_{\sigma}^{+}(x) - w_{\sigma}^{+}(y)]}{|x-y|^{2}}dy dx\\
&- \frac{1}{2\pi}\iint_{\R^2\setminus (a,b)^2}\frac{[\tau u_\sigma(x) - \tau u_\sigma(y)][w_{\sigma}^{+}(x) - w_{\sigma}^{+}(y)]}{|x-y|^{2}}dy dx\\
&+ \int_{\R \setminus (a,b)} w_\sigma(x)w_\sigma^+dx- \int_{\R\setminus (a,b)}Q(x)f(w_\sigma)w_\sigma^+dx.
\end{aligned}
$$ 
In the sequel we denote by $\Lambda_\sigma^1$ and $\Lambda_\sigma^2$ the following numbers
$$
\Lambda_\sigma^1 = \frac{1}{2\pi}\iint_{\R^2\setminus (a,b)^2}\frac{[\frac{u_1}{2}(x) - \frac{u_1}{2}(y)][w_{\sigma}^{+}(x) - w_{\sigma}^{+}(y)]}{|x-y|^{2}}dy dx
$$
and
$$
\Lambda_\sigma^2 = \frac{1}{2\pi}\iint_{\R^2\setminus (a,b)^2}\frac{[\tau u_\sigma(x) - \tau u_\sigma(y)][w_{\sigma}^{+}(x) - w_{\sigma}^{+}(y)]}{|x-y|^{2}}dy dx.
$$
Now, we analyze the behavior of $\Lambda_\sigma^1$ as $|\sigma|$ is large enough. Since
$$
w_{\sigma}^{+}(x) \to \frac{1}{2}u_1(x)\;\;\mbox{a.e.}\;\,x\in \R\setminus (a,b)\;\;\mbox{as $|\sigma| \to \infty$},
$$
and 
$$
\|w_{\sigma}^{+}\|_{H_{\tilde{\Omega}}^{1/2}} \leq \|w_\sigma\|_{H_{\tilde{\Omega}}^{1/2}} < \infty,
$$
we can conclude that
$$
w_{\sigma}^{+} \rightharpoonup  \frac{1}{2}u_1\;\;\mbox{as}\;\;|\sigma|\to \infty,
$$
which implies 
\begin{equation}\label{N16}
\Lambda_\sigma^1 \to \frac{1}{2\pi}\iint_{\R^2\setminus (a,b)^2} \frac{|\frac{1}{2}u_1(x) - \frac{1}{2}u_1(y)|^2}{|x-y|^2}dy dx\;\;\mbox{as}\;\;|\sigma|\to \infty. 
\end{equation}
On the other hand, setting  $h_\sigma (x) = \frac{1}{2}u_1(x+\sigma) - \tau u_\infty(x)$, we obtain 
$$
\Lambda_{\sigma}^{2} = \frac{1}{2\pi}\iint_{\R^2\setminus (a-\sigma, b-\sigma)^2} \frac{[\tau u_\infty(x) - \tau u_\infty(y)][h_{\sigma}^{+}(x) - h_{\sigma}^{+}(y)]}{|x-y|^{2}}dy dx.
$$
For each $T>0$ fixed, let us denote by $\mathbb{B}(0, T) = \{(x,y)\in \R^2:\;\;|(x,y)|< T\}$, the ball in $\R^2$ with center $(0,0)$ and radius $T>0$. In the following, we take $|\sigma|$ large enough such that 
$$
\mathbb{B}(0,T) \subset \R^2 \setminus (a-\sigma, b-\sigma)^2
$$
and the decomposition
$$
\R^2\setminus (a-\sigma, b-\sigma)^2 = \mathbb{B}(0,T) \cup \R^2 \setminus [\mathbb{B}(0,T) \cup (a-\sigma, b-\sigma)^2].
$$
From this, it follows that 
$$
\Lambda_{\sigma}^{2} = \Gamma_{\sigma}^{1} + \Gamma_{\sigma}^{2},
$$
where
$$
\begin{aligned}
\Gamma_{\sigma}^{1} & = \frac{1}{2\pi}\iint_{\mathbb{B}(0,T)}\frac{[\tau u_\infty(x) - \tau u_\infty(y)][h_{\sigma}^{+}(x) - h_{\sigma}^{+}(y)]}{|x-y|^{2}}dy dx;\\
\Gamma_{\sigma}^{2} &= \frac{1}{2\pi}\iint_{\R^2\setminus [\mathbb{B}(0,T)\cup (a-\sigma, b-\sigma)^2]} \frac{[\tau u_\infty(x) - \tau u_\infty(y)][h_{\sigma}^{+}(x) - h_{\sigma}^{+}(y)]}{|x-y|^{2}}dy dx.
\end{aligned}
$$
Following the same steps of (\ref{N16}), we have 
\begin{equation}\label{N17}
\Gamma_{\sigma}^{1} \to 0\;\;\mbox{as}\;\;|\sigma|\to \infty.
\end{equation}
Now, since $u_\infty\in H^{1/2}(\R)$, by H\"older inequality, 
\begin{equation}\label{N18}
\begin{aligned}
\Gamma_{\sigma}^{2} 
&\leq \frac{1}{2\pi}\left( \iint_{\R^2\setminus [\mathbb{B}(0,T)\cup (a-\sigma, b-\sigma)^2]}  \frac{|\tau u_\infty(x) - \tau u_\infty(y)|^2}{|x-y|^2}\right)^{1/2} \left(\iint_{\R^2\setminus [\mathbb{B}(0,T)\cup (a-\sigma, b-\sigma)^2]} \frac{|h_{\sigma}^{+}(x) - h_{\sigma}^{+}(y)|^2}{|x-y|^2}dy dx \right)^{1/2}\\
&\leq \frac{1}{2\pi}\left( \iint_{\R^2\setminus \mathbb{B}(0,T)}  \frac{|\tau u_\infty(x) - \tau u_\infty(y)|^2}{|x-y|^2}\right)^{1/2} \left(\iint_{\R^2\setminus (a-\sigma, b-\sigma)^2} \frac{|h_{\sigma}^{+}(x) - h_{\sigma}^{+}(y)|^2}{|x-y|^2}dy dx \right)^{1/2}\\
&\to 0\;\;\mbox{as}\;\;T\to \infty.
\end{aligned}
\end{equation}
From (\ref{N17}) and (\ref{N18}),  
\begin{equation}\label{N19}
\Lambda_{\sigma}^{2} \to 0\;\,\mbox{as}\;\;|\sigma| \to \infty.
\end{equation}
On the other hand, note that 
$$
\int_{\R \setminus (a,b)} w_\sigma(x)w_{\sigma}^+(x)dx = \int_{\R}\chi_{\Omega_\sigma}(x)w_\sigma^2(x)dx,
$$ 
where $\Omega_\sigma = \{x\in \R\setminus (a,b):\;\,\frac{1}{2}u_1(x) - \tau \tilde{u}(x-\sigma)>0\}$. Since $u_\infty(x-\sigma) \to 0$ a.e. in $\R$ as $|\sigma|\to \infty$ and $u_1>0$, we have 
$$
\chi_{\Omega_\sigma}(x)\to 1\;\;\mbox{a.e. in $\R\setminus (a,b)$}.
$$
Furthermore 
$$
\begin{aligned}
& \chi_{\Omega_\sigma}(x)w_\sigma^2(x)\leq w_\sigma^2(x)\in L^1(\R \setminus (a,b)),\;\;\mbox{and}\\ 
& \chi_{\Omega_\sigma}(x) w_{\sigma}^{2} \to (\frac{1}{2}u_1(x))^2\;\;\mbox{a.e. in $\R\setminus (a,b)$}.
\end{aligned}
$$
Thereby, by Lebesgue's Theorem, 
\begin{equation}\label{N20}
\int_{\R \setminus (a,b)} w_\sigma(x)w_{\sigma}^{+}(x)dx \to \int_{\R \setminus (a,b)} (\frac{1}{2}u_1(x))^2dx\;\;\mbox{as $|\sigma|\to \infty$}.
\end{equation} 
Finally, we are going to show that 
\begin{equation}\label{N21}
\int_{\R\setminus (a,b)}Q(x)f(w_\sigma)w_\sigma^+dx \to \int_{\R\setminus (a,b)}Q(x)f(\frac{u_1}{2})\frac{u_1}{2}dx\;\,\mbox{as}\;\;|\sigma|\to \infty.
\end{equation}
In fact, note that for $\tau \in [\frac{1}{2}, 2]$, since the function $t\to f(t)t$ is increasing in $[0,+\infty]$ and $u_\infty(x) \to 0$ a.e. in $\R$, we get
$$
Q(x)f(w_\sigma)(w_\sigma)^+ \leq Q(x)f(\frac{u_1}{2})\frac{u_1}{2} \in L^1(\R\setminus (a,b))
$$
and
$$
Q(x)f(w_\sigma(x))w_\sigma(x) \to Q(x)f(\frac{u_1}{2}(x))\frac{u_1}{2}(x)\;\;\mbox{a.e. $x\in \R\setminus (a,b)$}, \quad \mbox{as} \quad |\sigma| \to +\infty.
$$
In the above limit we are using strongly the fact that $u_\infty(x) \to 0$ when $|x| \to +\infty$. Hence, by Lebesgue's Theorem, 
\begin{equation}\label{N23}
\int_{\R\setminus (a,b)}Q(x)f(w_\sigma)w_{\sigma}dx \to \int_{\R\setminus (a,b)}Q(x)f(\frac{u_1}{2})\frac{u_1}{2}dx\;\;\mbox{as}\;\;|\sigma|\to \infty.
\end{equation}
In the same way, we can show that 
\begin{equation}\label{N24}
\int_{\{\frac{u_1}{2} - \frac{u_\sigma}{2}\geq 0\}} Q(x)f(w_\sigma)w_\sigma dx \to \int_{\R^\setminus (a,b)}Q(x)f(\frac{u_1}{2})\frac{u_1}{2}dx\;\;\mbox{as $|\sigma| \to \infty$}.
\end{equation}
Combining (\ref{N23}) with (\ref{N24}), we get (\ref{N21}). 
Therefore, by (\ref{N16}), (\ref{N19}), (\ref{N20}) and (\ref{N21})
$$
h_{\sigma}^{+}(\frac{1}{2}, \tau) = I'(w_\sigma)w_\sigma^+ \to I'(\frac{u_1}{2})\frac{u_1}{2}>0\;\;\mbox{as}\;\;|\sigma|\to +\infty \quad \mbox{and} \quad  \tau \in [\frac{1}{2},2].
$$ 
From this, there is $\sigma_0>0$ large enough such that 
$$
h_{\sigma}^{+}(\frac{1}{2}, \tau)>0\;\;\mbox{for $\tau \in [\frac{1}{2}, 2]$ and $|\sigma|>\sigma_0$}.
$$
In the same way, we can show that $h_{\sigma}^{+}(2,\tau) < 0$ and (\ref{N15}) holds.

\noindent 
From (\ref{N14}) and (\ref{N15}), we can apply the Mean Value Theorem due to Miranda \cite{CM} to find $\alpha^*, \tau^*\in [\frac{1}{2}, 2]$, which depend on $\sigma$, such that 
$$
h_{\sigma}^{\pm}(\alpha^*, \tau^*) = 0\;\;\mbox{for any}\;\;|\sigma|\geq \sigma_0.
$$
Thus,
$$
\alpha^*u_1 - \tau^*u_\sigma\in \mathcal{M},\;\;\mbox{for}\;\;|\sigma|\geq \sigma_0.
$$
By the definition of $c$, it suffices to show that 
\begin{equation}\label{N25}
\sup_{\alpha, \tau \in [\frac{1}{2}, 2]}I(\alpha u_1 - \tau u_\sigma) < c_1 + c_\infty\;\;\mbox{for}\;\;|\sigma|\geq \sigma_0.
\end{equation}
We would like to emphasize that the last inequality holds for $\sigma$ large, then by $(Q_2)$, it is enough to consider $R$ large enough to have $|\sigma|$ large enough.

Indeed, by Lemma \ref{alves},
$$
\begin{aligned}
I(\alpha u_1 - \tau u_\sigma) 
&\leq \frac{1}{2}\left(\frac{1}{2\pi}\iint_{\R^2\setminus (a,b)^2} \frac{|(\alpha u_1)(x) - (\alpha u_1)(y)|^2}{|x-y|^2}dy dx + \int_{\R\setminus (a,b)} (\alpha u_1)^2(x)dx\right)\\
&+\frac{1}{2}\left(\frac{1}{2\pi}\iint_{\R^2} \frac{|(\tau u_\sigma)(x) - (\tau u_\sigma)(y)|^2}{|x-y|^2}dy dx + \int_{\R} (\tau u_\sigma)^2(x)dx\right)\\
&- \int_{\R\setminus (a,b)} Q(x)[F(\alpha u_1) + F(\tau u_\sigma) - 2(f(\alpha u_1)\tau u_\sigma + f(\tau u_\sigma)\alpha u_1)]dx\\
&\leq I(\alpha u_1) + I_\infty(\tau u_\sigma) - \int_{\R \setminus (a,b)} [Q(x) - \tilde{Q}]F(\tau u_\sigma)dx\\
& + C_1\int_{\R\setminus (a,b)} (f(\alpha u_1)\tau u_\sigma + f(\tau u_\sigma)\alpha u_1)dx + \int_{a}^{b}\tilde{Q}F(\tau u_\sigma)dx.
\end{aligned}
$$
Therefore, 
\begin{equation}\label{N26}
\begin{aligned}
\sup_{\alpha, \tau \in [\frac{1}{2}, 2]} I(\alpha u_1 - \tau u_\sigma)&\leq \sup_{\alpha \geq 0}I(\alpha u_1) + \sup_{\tau \geq 0} I(\tau u_\sigma) - \int_{\R \setminus (a,b)} [Q(x) - \tilde{Q}]F( \frac{1}{2}u_\sigma)dx\\
& + C_1\int_{\R\setminus (a,b)} (f(\alpha u_1)\tau u_\sigma + f(\tau u_\sigma)\alpha u_1)dx + \int_{a}^{b}\tilde{Q}F(2 u_\sigma)dx.
\end{aligned}
\end{equation} 
Now, from ($f_5$), $(Q_2)$ and (\ref{decaimento}), we get   
$$
\begin{aligned}
\int_{\R \setminus (a,b)} [Q(x) - \tilde{Q}]F( \frac{1}{2}u_\sigma)dx &\geq \frac{CC_p}{p} \int_{\R \setminus (a-\sigma , b-\sigma)} [Q(x+\sigma) - \tilde{Q}]{u_{\infty}^p(x)} dx\\
&\geq CR^{\gamma}\int_{R}^{2R}u_{\infty}^p(x)\, dx\geq  C_1 R^{\gamma-2p+1}\\
\end{aligned}
$$
for $R$ large enough. By (\ref{GS05}),$(Q_2)$ and using the fact that $u_\infty \in L^{\infty}(\R)$, we get  
$$
\begin{aligned}
\int_{a}^{b} \tilde{Q}F(2u_\sigma)dx &\leq \tilde{Q}\int_{a}^{b} \frac{\epsilon}{q+1}|2u_\sigma(x)|^{q+1} dx +\tilde{Q}C_\epsilon \int_{a}^{b} |2u_\sigma(x)|\left(e^{\pi \nu u_\sigma^2(x)} -1 \right)dx\\
&\leq {C}_2 (b-a).
\end{aligned}
$$
From ($f_1$) and ($f_2$), we have 
\begin{equation}\label{N27}
\begin{aligned}
\int_{\R \setminus (a,b)} f(\alpha u_1)\tau u_\sigma dx & \leq \epsilon\tau \alpha^q \int_{\R \setminus (a,b)} |u_1(x)|^{q}u_\sigma (x)dx + C_\epsilon \tau \int_{\R \setminus (a,b)} u_\sigma (x)u_1^{q}(x) \left(e^{\pi (\alpha u_1)^2} -1 \right)dx.
\end{aligned}
\end{equation}
By H\"older inequality,  
\begin{equation}\label{N28}
\begin{aligned}
\int_{\R \setminus (a,b)} u_1^q(x)u_\infty (x-\sigma)dx &\leq \left(\int_{\R\setminus (a,b)} u_1^{q+1}(x)dx \right)^{\frac{q}{q+1}} \left( \int_{\R} u_{\infty}^{q+1}(x-\sigma)dx\right)^{\frac{1}{q+1}}\\
&\leq C^{q} \|u_1\|^{\frac{q(q-1)}{q+1}}_{L^{\infty}(\R\setminus (a,b))}\|u_1\|^{\frac{2q}{q+1}}_{L^{q2}(\R\setminus (a,b))}\|u_\infty\|_{L^{q+1}(\R)}\leq C_3.
\end{aligned}
\end{equation}
Since $\|u_1\|_{L^{\infty}(\R\setminus (a,b))}$ and $\|u_1\|_{L^{2}(\R\setminus (a,b))}$ are bounded from above by a constant that is independent of $R$, the constant  $C_3$ can be choose independent of $R$, for more details see Lemma \ref{Rlm01}. 

In a similar way, we have   
$$
\int_{\R \setminus (a,b)} u_\sigma (x)u_1^{q}(x) \left(e^{\pi (\alpha u_1)^2} -1 \right)dx \leq C_4, 
$$
where $C_5$ is a constant independent of $R$.  Therefore, 
\begin{equation}\label{N29}
\int_{\R \setminus (a,b)} f(\alpha u_1)\tau u_\sigma dx  \leq C_5
\end{equation}
where $C_5$ is a constant independent of $R$. A similar argument works to show that 
\begin{equation}\label{N30}
\int_{\R\setminus (a,b)}f(\tau u_\sigma)\alpha u_1dx \leq C_6
\end{equation}
where $C_6$ is a constant independent of $R$. From the above analysis, we derive that 
$$
\sup_{\alpha, \tau \in [\frac{1}{2}, 2]}I(\alpha u_1 - \tau u_\sigma) \leq \sup_{\alpha \geq 0} I(\alpha u_1) + \sup_{\tau \geq 0} I(\tau u_\sigma) + R^{\gamma-2p+1} \left(-C + \frac{C_7}{R^{\gamma-2p +1} }\right)
$$
where $C_7$ is a constant independent of $R$. Hence, for $R$ large enough  
\begin{equation}\label{N31}
\sup_{\alpha, \tau \in [\frac{1}{2}, 2]} I(\alpha u_1 - \tau u_\sigma) < c_1 + c_\infty,
\end{equation}  
showing that $c < c_1 + c_\infty.$
\end{proof}

Let us introduce the function $\varphi_\rho(x) = \varphi (\frac{x}{\rho})$, where $\rho\gg |b-a|$ and  
$$
\varphi (x) = \begin{cases}
1,&|x|\leq 1\\
0,&|x| \geq 2,
\end{cases}
$$
and consider the following fractional problem 
$$
\left\{
\begin{aligned}
(-\Delta)^su + u &= \varphi_\rho (x)Q(x)f(u),\;\;\mbox{in}\:\:\R\setminus (a,b)\\
\mathcal{N}_{1/2}u(x) &= 0\;\;\mbox{in}\;\;(a,b).
\end{aligned}
\right.
\eqno{(P_\rho)}
$$
Associated to problem $(P_\rho)$ we have the energy functional  
$$
I_\rho(u) = \frac{1}{2}\|u\|_{H_{\tilde{\Omega}}^{1/2}}^2 - \int_{\R\setminus (a,b)} Q(x)\varphi_\rho(x)f(u)dx.
$$
Moreover, we introduce the nodal set
$$
\mathcal{M}_\rho = \{u\in \mathcal{N}_\rho:\;\;u^{\pm} \not \equiv 0\;\;\mbox{and}\;\;I'_{\rho}(u)u^{\pm} = 0\}
$$
with
$$
\mathcal{N}_{\rho} = \{u\in H_{\tilde{\Omega}}^{s}\setminus \{0\}:\;\;I'_\rho(u)u=0\}
$$
and the number
$$
c_\rho = \inf_{u\in \mathcal{M}_\rho} I_\rho(u).
$$
By repeating a similar reasoning as used in \cite{KTKWRW}( see also \cite{AlvesSouto}), we can show that, for each $\rho \gg |b-a|$ there exists $u_\rho \in \mathcal{M}_\rho$ such that $u_{\rho}^{\pm} \not \equiv 0$ and $c_\rho = I_\rho(u_\rho)$. 

\begin{lemma}\label{Nlm03}
$$
\lim_{\rho\to +\infty}c_\rho = c = \inf_{u\in \mathcal{M}}I(u).
$$
\end{lemma}
\begin{proof}
Note that $I(u) \leq I_\rho(u)$, for all $u\in H_{\tilde{\Omega}}^{1/2}$. For each $u^{\pm}\in \mathcal{M}_\rho$, there exist $t_\rho, s_\rho>0$ such that 
$$
t_\rho u^+ + s_\rho u^{-} \in \mathcal{M}.
$$
Thus
$$
c\leq I(t_\rho u^+ + s_\rho u^{-}) \leq I_\rho (t_\rho u^+ + s_\rho u^{-}) \leq I_\rho(u) \leq c_\rho, \quad \forall \rho \gg |b-a|,
$$
and 
\begin{equation}\label{N32}
c\leq \liminf_{\rho\to +\infty}c_\rho.
\end{equation}
On the other hand, given $w\in \mathcal{M}$, there exist $t_\rho, s_\rho>0$ such that 
$$
t_\rho w^+ + s_\rho w^-\in \mathcal{M}_\rho.
$$ 
A direct computation gives that $(t_\rho)$ and $(s_\rho)$ are bounded, because $I_\rho(t_\rho w^+ + s_\rho w^-) \geq c_\rho >0$, and we have the limit below  
$$
\lim_{\rho \to +\infty}I_\rho(t_\rho w^+ + s_\rho w^-) <0,
$$
if $t_\rho \to +\infty$ or $s_\rho \to +\infty$ as $\rho \to +\infty$. Hence, by Lebesgue's theorem,
$$
\int_{\R \setminus (a,b)} (1-\varphi_\rho (x))Q(x)F(t_\rho w^+ + s_\rho w^-)dx \to 0\;\;\mbox{as}\;\;\rho\to +\infty
$$ 
from where it follows that
$$
c_\rho \leq I(w) + o_\rho(1).
$$
Consequently 
$$
\limsup_{\varrho\to +\infty}c_\varrho\leq I(w),\;\;\forall w\in \mathcal{M}
$$
leading to 
\begin{equation}\label{N33}
\limsup_{\rho\to +\infty}c_\rho \leq c.
\end{equation}
By (\ref{N32}) and (\ref{N33}), 
$$
\lim_{\rho\to \infty}c_\rho = c.
$$
\end{proof}

\noindent 
{\bf Proof of Theorem \ref{main2}:} In what follows, we set $\rho_n \to +\infty$ and $u_n=u_{\rho_n}$. Since 
$$
c + \|u_n\|_{H_{\tilde{\Omega}}^{1/2}}\geq I_{\rho_n}(u_n) - \frac{1}{\theta}I'_{\rho_n}(u_n)u_n = \left( \frac{1}{2} - \frac{1}{\theta}\right)\|u_n\|_{H_{\tilde{\Omega}}^{1/2}}^{2}, 
$$
we conclude that $(u_n)$ is bounded in $H_{\tilde{\Omega}}^{1/2}$. Then, for some subsequence, there is $u\in H_{\tilde{\Omega}}^{1/2}$ such that 
$$
u_n  \rightharpoonup u\;\;\mbox{in}\;\;H_{\tilde{\Omega}}^{1/2} \quad \mbox{and}\quad I'(u)=0. 
$$ 
Now we are going to show that $u^{\pm} \neq 0$. Indeed we need to consider three cases:
\begin{itemize}
\item[(i)] $u^+ = u^-=0$.
\item[(ii)] $u^+\neq 0$ and $u^- = 0$.
\item[(iii)] $u^+ = 0$ and $u^- \neq 0$
\end{itemize}
We will prove that the above cases do not hold, therefore $u^{\pm} \neq 0$. We only prove $(i)$, since the other cases follow with the same type of arguments.

By Proposition \ref{GSresult03}, there exist $\eta, \kappa>0$ and sequences $(y_{n}^{1})$ and $(y_{n}^{2})$ in $\R\setminus (a,b)$ with $|y_n^1|, |y_{n}^{2}|\to \infty$ such that   
\begin{equation}\label{N34}
\liminf_{\varrho\to +\infty} \int_{\Lambda(y_{n}^{1}, \kappa)} |u_{n}^{+}|^2dx\geq \eta\quad \mbox{and}\quad \liminf_{n \to +\infty} \int_{\Lambda(y_{n}^{2}, \kappa)} |u_{n}^{-}|^2dx\geq \eta,
\end{equation}
where $\Lambda(y,\kappa) = (-\kappa,\kappa) \cap (\R\setminus (a,b))$. Let $w_n (x) = u_{\rho}(x+y_{n}^{1})$ and $z_{n}(x) = u_n(x+y_{\rho}^{2})$. Arguing as in the proof of Theorem \ref{main1}, we can assume that there exist $w,z\in H^{1/2}(\R)\setminus \{0\}$ such that $w_n  \rightharpoonup w$ and $z_n  \rightharpoonup z$ in $H^{1/2}(-T,T)$ for all $T>0$,  with $w^+\neq 0$ and $z^-\neq 0$. Now, let $\psi \in H_{\tilde{\Omega}}^{1/2}$ be a test function with bounded support. Since $I'_{\rho_n}(u_{n}) = 0$, then   
\begin{equation}\label{N35}
I'_{\rho_n}(u_n)\psi(.-y_{n}^{1}) = 0.
\end{equation}
By doing the change of variable $\tilde{x} = x - y_{n}^1$ and $\tilde{y} = y - y_{n}^{1}$, from (\ref{N35}) we get 
\begin{equation}\label{N36}
\begin{aligned}
&\frac{1}{2\pi}\iint_{\R^{2}\setminus (a-y_{n}^{1}, b-y_{n}^{1})^2} \frac{[w_{n}(x) - w_{n}(y)][\psi(x) - \psi(y)]}{|x-y|^{N+2s}}dy dx + \int_{\R \setminus (a-y_{n}^{1}, b-y_n^1)} w_{n}(x)\psi(x)dx \\
&\hspace{6cm}= \int_{\R\setminus (a-y_{n}^{1}, b-y_n^1)} \varphi_{\rho}(x+y_{n}^{1})Q(x + y_{n}^{1})f(w_{n}^{1})\psi(x)dx.
\end{aligned}
\end{equation} 
By the weak convergence of $(w_n)$ to $w$ in $H^{1/2}(-T,T)$, we have 
\begin{equation}\label{N37}
\begin{aligned}
\frac{1}{2\pi}\iint_{\R^{2}\setminus (a-y_n^1, b-y_n^1)^2} &\frac{[w_{n}(x) - w_{n}(y)][\psi(x) - \psi(y)]}{|x-y|^{N+2s}}dy dx + \int_{\R \setminus (a-y_{n}^{1}, b - y_n^1)} w_{n}(x)\psi(x)dx \\
&\to \frac{1}{2\pi} \iint_{\R^{2}} \frac{[w(x) - w(y)][\psi(x) - \psi(y)]}{|x-y|^{N+2s}}dydx + \int_{\R} w(x)\psi(x)dx.
\end{aligned}
\end{equation}
Moreover, again by Lebesgue's theorem, 
\begin{equation}\label{N38}
\int_{\R\setminus (a-y_{n}^{1}, b-y_n^1)} \varphi_n(x+y_{n}^{1})Q(x+y_{n}^{1})f(w_n)\psi(x)dx \to \int_{\R^N}\tilde{Q}f(w)\psi(x)dx
\end{equation}
Combining (\ref{N37}) and (\ref{N38}), we obtain
$$I'_\infty(w)\psi = 0.$$
In the same way we can show that $I'_\infty(z)\psi = 0$. 

Now by (\ref{N03}) and (\ref{N04}) and the above equality, we have
$$
I'_\infty(w^{+})w^{+} = I'_\infty(w)w^+ + \frac{1}{2\pi}\iint_{\R^{2N}}\frac{w^{+}(x)w^{-}(y) + w^{+}(y)w^{-}(x)}{|x-y|^{N+2s}}dy dx\leq 0
$$
and 
$$
I'_\infty(z^{-})z^{-} = I'_\infty(z)z^- + \frac{1}{2\pi}\iint_{\R^{2N}}\frac{z^{+}(x)z^{-}(y) + z^{+}(y)z^{-}(x)}{|x-y|^{N+2s}}dy dx\leq 0.
$$
So there are $t_{w}, t_z \in (0,1]$ such that $t_{w}w^{+}, t_{z} z^{-} \in \mathcal{N}_\infty$. Thus, by Fatou's lemma, Lemma \ref{Nlm02} and Lemma \ref{Nlm03} we have
$$
\begin{aligned}
2c_{\infty} &\leq I_{\infty}(t_ww^+) + I_{\infty}(t_zz^-)\\
& = \left[I_\infty(t_ww^+) - \frac{1}{\theta}I'_\infty(t_ww^+)t_ww^+\right] + \left[I(t_zz^-) - \frac{1}{\theta}I'_\infty(t_zz^-)t_zz^-\right]\\
& = \liminf_{n \to +\infty} \left[I_{\rho_n}(u_n) - \frac{1}{2}I'_{\rho_n}(u_n)u_n\right] = \lim_{n \to +\infty} I_{\rho_n}(u_n)= \lim_{n \to +\infty} c_{\rho_n}=c < c_1+c_\infty,
\end{aligned}
$$
which is absurd.

\section{Appendix}

In this section, our main goal is to study some $L^\infty$ estimate and decay at infinite of the ground state solution $u$ of ($P$) that was obtained in Section 3. 
We start our analysis with the following lemma.
\begin{lemma}\label{Rlm01}
The ground state solution  $u_1$ belongs to $L^\infty(\R\setminus (a,b))$. Moreover, $\|u_1\|_{L^\infty(\R\setminus (a,b))} \leq M$ , for some $M$ that is independent of $R >|a|+|b|+1$.
\end{lemma}
\begin{proof} In what follows $u$ denotes $u_1$. In this proof we adapt for our case some arguments found in \cite[Lemma 5.4]{AlvesAmbrosio}.
For all $t\in \R$ and $L>0$, we set 
\begin{equation}\label{A01}
t_L = \mbox{sgn}(t) \min\{|t|, L\}.
\end{equation}
By \cite[Lemma 3.1]{AISMMS}, for all $a,b\in \R$, $\beta > 1$ and $L>0$ we have 
\begin{equation}\label{A02}
(a-b)(a|a|_{L}^{2(\beta -1)} - b|b|_{L}^{2(\beta -1)}) \geq \frac{2\beta -1}{\beta^2} (a|a|_{L}^{\beta - 1} - b|b|_{L}^{\beta -1})^2.
\end{equation}
Since the mapping 
$t\to t|t|_{L}^{2(\beta-1)} \;\,\mbox{is Lipschitz in $\R$},$  
then $uu_{L}^{2(\beta-1)} \in H_{\tilde{\Omega}}^{1/2}$. Taking $v = uu_{L}^{2(\beta-1)}$ as a test function in $(P)$, we have 
\begin{equation}\label{A03}
\begin{aligned}
\frac{1}{2\pi}\iint_{\R^2\setminus (a,b)^2} \frac{[u(x) - u(y)][v(x) - v(y)]}{|x-y|^{2}}dy dx + \int_{\R\setminus (a,b)} u(x)v(x)dx = \int_{\R\setminus (a,b)} Q(x)f(u(x))v(x)dx.
\end{aligned}
\end{equation}
By ($f_1$) and ($f_3$), given $\tau >1$ and $\epsilon >0$, there exists $C_\epsilon>0$ such that 
$$
|f(s)| \leq \epsilon |s| + C_\epsilon |s|\left( e^{\pi \tau s^2} - 1\right)\;\;\forall s\in \R.
$$ 
So
\begin{equation}\label{A04}
\begin{aligned}
\int_{\R\setminus (a,b)} Q(x)f(u(x))v(x)dx &\leq \epsilon \int_{\R\setminus (a,b)} Q(x) u(x)v(x) dx + C_\epsilon \int_{\R\setminus (a,b)} Q(x)u(x)\left(e^{\pi \tau u^2} - 1\right)v(x)dx\\
&\leq \epsilon \|Q\|_{\infty}\int_{\R\setminus (a,b)}u(x)v(x)dx + C_\epsilon \|Q\|_{\infty}\int_{\R\setminus (a,b)} u(x)v(x)\left( e^{\pi \tau u^2} - 1\right)dx. 
\end{aligned}
\end{equation}
Combining (\ref{A03}) with (\ref{A04}) we get
$$
\begin{aligned}
\frac{1}{2\pi}\iint_{\R^2\setminus (a,b)^2} \frac{[u(x) - u(y)][v(x) - v(y)]}{|x-y|^{2}}dy dx &+ (1-\epsilon \|Q\|_\infty)\int_{\R\setminus (a,b)}u(x)v(x)dx\\
& \leq C_\epsilon \|Q\|_{\infty}\int_{\R\setminus (a,b)} u(x)v(x)\left( e^{\pi \tau u^2} -1\right)dx.
\end{aligned}
$$
Taking $\epsilon>0$ small enough such that $1-\epsilon \|Q\|_\infty>0$ we get
\begin{equation}\label{A05}
\begin{aligned}
(1-\epsilon\|Q\|_\infty)\left(\frac{1}{2\pi}\iint_{\R^2\setminus (a,b)^2} \frac{[u(x) - u(y)][v(x) - v(y)]}{|x-y|^{2}}dy dx\right. & + \left.\int_{\R\setminus (a,b)}u(x)v(x)dx \right)\\
& \leq C_\epsilon \|Q\|_{\infty}\int_{\R\setminus (a,b)} u(x)v(x)\left( e^{\pi \tau u^2} - 1\right)dx.
\end{aligned}
\end{equation} 
Now we can see that $h(x) = \left( e^{\pi \tau u^2} - 1\right) \in L^q(\R)$ for some $q>1$ close to $1$, with $\tau, q>1$ are such that $\tau q \|u\|_{H_{\tilde{\Omega}}^{1/2}}^{2} < 1$ and 
$$
\|h\|_{L^q(\R)} \leq C.
$$
We would like point out that $C$ is independent of $R>|a|+|b|+1$, because the constant $\xi$ in Proposition \ref{GSresult05} is independent of $R>|a|+|b|+1$.

Thereby, by Lemma \ref{lm03}, (\ref{A05}) and H\"older inequality  
$$
\begin{aligned}
&\|uu_{L}^{\beta-1}\|_{L^{\gamma}(\R\setminus (a,b))}^{2} \leq S \|uu_{L}^{\beta-1}\|_{H_{\tilde{\Omega}}^{1/2}}^{2}\\
& = S\left(\frac{1}{2\pi}\iint_{\R^{2}\setminus (a,b)^2} \frac{|(uu_{L}^{\beta-1})(x) - (uu_{L}^{\beta-1})(y)|^2}{|x-y|^{2}}dy dx + \int_{\R \setminus (a,b)} (uu_{L}^{\beta-1})^2(x)dx\right)\\
&\leq S\left( \frac{\beta^2}{2\beta-1} \frac{1}{2\pi}\iint_{\R^{2}\setminus (a,b)^2} \frac{(u(x) - u(y))(u(x)u_{L}^{2(\beta-1)}(x) - u(y)u_{L}^{2(\beta-1)}(y))}{|x-y|^{N+2s}} dy dx + \int_{\R\setminus (a,b)} u^2(x)u_{L}^{2(\beta-1)}(x)dx\right)\\
&\leq \frac{S\beta^2}{2\beta-1} \left(\frac{1}{2\pi}\iint_{\R^{2}\setminus (a,b)^2} \frac{(u(x) - u(y))(u(x)u_{L}^{2(\beta-1)}(x) - u(y)u_{L}^{2(\beta-1)}(y))}{|x-y|^{N+2s}} dy dx + \int_{\R\setminus (a,b)} u^2(x)u_{L}^{2(\beta-1)}(x)dx\right)\\
&\leq \frac{S\beta^2}{2\beta-1} \frac{C_\epsilon \|Q\|_\infty}{1-\epsilon \|Q\|_\infty} \int_{\R\setminus (a,b)} u(x)v(x)\left( e^{\pi \tau u^2} - 1\right)dx\\
&\leq \tilde{C} \beta^2 \int_{\R\setminus (a,b)} u(x)v(x)\left( e^{\pi \tau u^2} - 1\right)dx \leq \tilde{C}\beta^2 \left( \int_{\R\setminus (a,b)}(u(x)v(x))^{q'}dx\right)^{1/q'}\left(\int_{\R\setminus (a,b)} h^q(x)dx \right)^{1/q}\\
&\leq \tilde{C}_1 \beta^2\left( \int_{\R\setminus (a,b)} (u(x)v(x))^{q'}dx\right)^{1/q'}
\end{aligned}
$$
Since $u_L\leq u$, the last inequality leads to  
\begin{equation}\label{A06}
\begin{aligned}
\|u_{L}^{\gamma \beta}\|_{L^{\gamma \beta}(\R\setminus (a,b))}^{2\beta} \leq \tilde{C}_1\beta^2 \left( \int_{\R\setminus (a,b)} u^{2\beta q'}(x)dx\right)^{1/q'} = \tilde{C}_1\beta^2 \|u\|_{L^{2\beta q'}(\R\setminus (a,b))}^{2\beta}
\end{aligned}
\end{equation}
By passing to the limit in (\ref{A06}) as $L \rightarrow +\infty$, the Fatou's Lemma gives
\begin{equation}\label{A07}
\|u\|_{L^{\gamma \beta}(\R\setminus (a,b))} \leq C^{\frac{1}{\beta}} \beta^{\frac{1}{\beta}} \|u\|_{L^{2\beta q'}(\R\setminus (a,b))}.
\end{equation}
whenever $u^{\gamma \beta} \in L^1(\R\setminus (a,b))$. 
Now we claim that 
\begin{equation}\label{A08}
\|u\|_{L^{\chi^{k+1}2q'}(\R\setminus (a,b))} \leq C^{\frac{1}{\chi^k} +\cdots + \frac{1}{\chi}}\chi^{\frac{k}{\chi^k}+\cdots + \frac{1}{\chi}}\|u\|_{L^{\gamma}(\R\setminus (a,b))},
\end{equation}
where 
$$
\chi = \frac{\gamma}{2q'}
$$
with $\gamma > 2q'$ fixed. In fact, by (\ref{A07}), if $\beta = \chi$, then we have 
$$
u^\beta \in L^{2q'}(\R\setminus (a,b))
$$
and 
\begin{equation}\label{A09}
\|u\|_{L^{\chi^22q'}(\R\setminus (a,b))} \leq C^{\frac{1}{\chi}} \chi^{\frac{1}{\chi}} \|u\|_{L^{\gamma}(\R\setminus (a,b))}.
\end{equation}
Now, if $\beta = \chi^2$, by (\ref{A09}) implies that 
$$
u^{\beta} \in L^{2q'}(\R\setminus (a,b))
$$
and by (\ref{A07}) and (\ref{A09}) 
\begin{equation}\label{A10}
\|u\|_{L^{\chi^3 2q'}(\R\setminus (a,b))} \leq C^{\frac{1}{\chi^2}} \chi^{\frac{2}{\chi^2}} \|u\|_{L^{\chi^2 2q'}(\R\setminus (a,b))} \leq C^{\frac{1}{\chi^2} + \frac{1}{\chi}} \chi^{\frac{2}{\chi^2} + \frac{1}{\chi}} \|u\|_{L^{\gamma}(\R\setminus (a,b))}.
\end{equation}
It follows from (\ref{A10})
\begin{equation}\label{A11}
u^{\chi^3} \in L^{2q'}(\R\setminus (a,b)).
\end{equation}
Now by induction, we suppose that (\ref{A08}) holds for some $k>1$. Then
$$
u^\beta \in L^{2q'}(\R\setminus (a,b)),
$$
with $\beta = \chi^{k+1}$. Moreover by (\ref{A07}) and by hypotheses we get 
\begin{equation}\label{A12}
\|u\|_{L^{\chi^{k+2}2q'}(\R\setminus (a,b))} \leq C^{\frac{1}{\chi^{k+1}}}\chi^{\frac{k+1}{\chi^{k+1}}}\|u\|_{L^{\chi^{k+1}2q'}(\R\setminus (a,b))}\leq C^{\sum_{i=1}^{k}\frac{1}{\chi^i}} \chi^{\sum_{i=1}^{k} \frac{i}{\chi^i}}\|u\|_{L^{\gamma}(\R\setminus (a,b))},
\end{equation}
which proves (\ref{A08}). Taking the limit in (\ref{A08}) as $k\to \infty$ and recalling that $u\in H_{\tilde{\Omega}}^{1/2}$, we get that $u\in L^\infty(\R\setminus (a,b))$ with 
$$
\|u\|_{L^\infty(\R\setminus (a,b))} \leq C^{\frac{1}{\chi-1}} \chi^{\frac{\chi^2}{(\chi-1)^2}}\|u\|_{L^\gamma(\R\setminus (a,b))},
$$ 
where 
$$
\sum_{i=1}^{\infty} \frac{1}{\chi^i} = \frac{1}{\chi -1}\;\;\mbox{and}\;\;\sum_{i=1}^{\infty} \frac{i}{\chi^i} = \frac{\chi^2}{(\chi-1)^2}.
$$ 
By Sobolev embedding and Proposition \ref{GSresult05}, there is $M_1$ independent of $R>|a|+|b|+1$ such that 
$$
\|u\|_{L^\gamma(\R\setminus (a,b))} \leq M_1.
$$ 
Therefore,  there is $M$ independent of $R>|a|+|b|+1$ such that 
$$
\|u\|_{L^\infty(\R\setminus (a,b))} \leq M. 
$$ 
\end{proof}

Our next goal is showing that $u(x) \to 0$ as $|x| \to +\infty$. However, in order to prove this, we will firstly study some properties of the solution of the following linear problem. 
\begin{equation}\label{A13}
\left\{\begin{array}{l} 
\frac{1}{2}(-\Delta)^{1/2}v + v = g(x)\;\;\mbox{in}\;\;\R,\\
v\in H^{1/2}(\R),
\end{array}
\right.
\end{equation}
where 
$$
g(x) = Q(x)f(\tilde{u}(x)).
$$ 
and 
$$
\tilde{u}(x) = \begin{cases}
u(x),&x\in \R \setminus (a,b)\\
0,&x\in (a,b).
\end{cases}
$$
Note that $g\in L^2(\R^N)$. Consequently, by Riesz's Theorem,  problem (\ref{A13}) has a unique weak solution $v\in H^{1/2}(\R)$, which is given by   
\begin{equation}\label{A14}
v(x) = (\mathcal{K} * g)(x) = \int_{\R} \mathcal{K}(x-\xi)g(\xi)d\xi,
\end{equation} 
where $\mathcal{K}$ is the Bessel kernel
\begin{equation}\label{A15}
\mathcal{K}(x) = \mathcal{F}^{-1}\left( \frac{1}{1+\frac{1}{2}|\xi|}\right)(x)=2K_*(2x),
\end{equation}
where
$$
K_*(x)=\mathcal{F}^{-1}\left( \frac{1}{1+|\xi|}\right)(x).
$$
The function $\mathcal{K}$ verifies the following properties: 
\begin{itemize}
\item[$(K_1)$] $\mathcal{K} $ is positive, radially symmetric and smooth in $\mathbb{R}\setminus \{0\}$, \\
\item[$(K_2)$] There is $C>0$ such that  
$$
\mathcal{K}(x) \leq \frac{C}{|x|^{2}}, \quad \forall x \in \mathbb{R}\setminus \{0\} 
$$
\item[$(K_3)$] There is a constant $C$ such that 
$$
\mathcal{K}'(x) \leq \frac{C}{|x|^{3}}\;\;\mbox{if}\;\;|x|\geq 1.
$$ 
\item[$(K_4)$] $\mathcal{K} \in L^{q}(\mathbb{R}), \quad \forall q \in [1,\infty)$. 
\end{itemize}
The properties above mentioned were proved in \cite{PFAQJT} for function $K_*$, and so, they must hold for $\mathcal{K}$. Since $u(x)\geq 0$ for all $x\in \R\setminus \Omega$, $u \not \equiv 0$ and  $\mathcal{K}$ is positive, then $v(x)>0$ for all $x \in \mathbb{R}$. 

By using the above information, we are able to prove the following result
\begin{lemma}\label{continuity}
 The function $v$ is continuous, that is, $v\in C(\R)$.
 \end{lemma}
\begin{proof}
Let $\delta >0$, $x_0\in \R$ and $T > |x_0|+2 \delta$. For any $x\in (x_0-\delta, x_0+\delta)$, we have 
$$
\begin{aligned}
|v(x) - v(x_0)| &= \left|\int_{\R} \mathcal{K}(x-\xi)g(\xi)d\xi - \int_{\R} \mathcal{K}(x_0 - \xi)g(\xi)d\xi\right|\\
&\leq \int_{\R} |\mathcal{K}(x-\xi) - \mathcal{K}(x_0-\xi)||g(\xi)|d\xi \\
&= \int_{-T}^{T} |\mathcal{K}(x-\xi) - \mathcal{K}(x_0-\xi)||g(\xi)|d\xi + \int_{(-T,T)^c} |\mathcal{K}(x-\xi) - \mathcal{K}(x_0-\xi)||g(\xi)|d\xi
\end{aligned}
$$
Note that, by H\"older inequality,  
$$
\begin{aligned}
\int_{(-T,T)^c} |\mathcal{K}(x-\xi) - \mathcal{K}(x_0-\xi)||g(\xi)|d\xi & \leq \left( \int_{(-T,T)^c} |\mathcal{K}(x-\xi)- \mathcal{K}(x_0-\xi)|^2d\xi\right)^{1/2} \left( \int_{(-T,T)^c} |g(\xi)|^2d\xi\right)^{1/2}.
\end{aligned}
$$ 
Since $\mathcal{K}$ is smooth, there exists $C>0$ 
$$
\begin{aligned}
|\mathcal{K}(x - \xi) - \mathcal{K}(x_0 -\xi)| &\leq |\mathcal{K}' (x_0-\xi + \theta (x-x_0))||x-x_0|\\
&\leq C \frac{1}{|x_0-\xi + \theta(x-x_0)|^{3}} |x-x_0|\\
&\leq C \frac{|x-x_0|}{|\xi|^{3}}.
\end{aligned}
$$ 
Then
$$
\begin{aligned}
\int_{(-T,T)^c} |\mathcal{K}(x-\xi) - \mathcal{K}(x_0 - \xi)|^2 d\xi &\leq \tilde{C} |x-x_0|^2 \int_{(-T,T)^c} \frac{d\xi}{|\xi|^{6}}
= \tilde{C}\delta^2 \frac{1}{T^5}.
\end{aligned}
$$
So
\begin{equation*}
\int_{(-T,T)^c} |\mathcal{K}(x-\xi) - \mathcal{K}(x_0-\xi)||g(\xi)|d\xi \leq  \tilde{C} \frac{\delta}{T^5} \left( \int_{\R} |g(\xi)|^2d\xi\right)^{1/2}.
\end{equation*}
Therefore, given $\epsilon$, we can fix $\delta$ small enough such that
\begin{equation}\label{A16}
\int_{(-T,T)^c} |\mathcal{K}(x-\xi) - \mathcal{K}(x_0-\xi)||g(\xi)|d\xi < \frac{\epsilon}{3}.
\end{equation}
On the other hand, fixing $q \in (1,\infty)$, $q'=\frac{q}{q-1}$ and using $(K_4)$, we obtain by  H\"older inequality  
\begin{equation*}
 \int_{x_0-\delta}^{x_0+\delta} |\mathcal{K}(x-\xi) - \mathcal{K}(x_0-\xi)||g(\xi)|d\xi \leq C\left(\int_{x_0-\delta}^{x_0+\delta} |g(\xi)|^{q'}\,d\xi\right)^{\frac{1}{q'}}.  
\end{equation*}
From this, we can fix $\delta>0$ small enough such that 
\begin{equation}\label{A17}
\int_{x_0-\delta}^{x_0+\delta} |\mathcal{K}(x-\xi) - \mathcal{K}(x_0-\xi)||g(\xi)|d\xi < \frac{\epsilon}{3}.  
\end{equation}
Finally, we can use the continuity of $K$ in $\mathbb{R} \setminus \{0\}$ to prove that 
\begin{equation}\label{A18}
\int_{(-T,T) \setminus (x_0-\delta, x_0+\delta)} |\mathcal{K}(x-\xi) - \mathcal{K}(x_0-\xi)||g(\xi)|d\xi < \frac{\epsilon}{3},  
\end{equation}
when $\delta$ is smaller enough. Now, the lemma follows from (\ref{A16})-(\ref{A18}).

\end{proof}

Our next lemma studies the behavior of $v$ at infinity. In this proof, we use some arguments developed in Alves and Miyagaki \cite[Lemma 2.6]{CAOM} ( see also \cite{Alves16}).   
\begin{lemma} \label{ZERO}
$$	
v(x) \to 0 \quad \mbox{as} \quad |x| \to +\infty.
$$
\end{lemma}
\begin{proof}
Given $\delta >0$, we have 
$$
\begin{aligned}
0\leq v(x) &\leq  \int_{\mathbb{R}}\mathcal{K}(x-y)Q(y)|f(\tilde{u})|dy\\
&= \left(\int_{-\infty}^{x-\frac{1}{\delta}} + \int_{x + \frac{1}{\delta}}^{\infty} \right) \mathcal{K}(x-y)Q(y)|f(\tilde{u})|dy + \int_{x-\frac{1}{\delta}}^{x+\frac{1}{\delta}} \mathcal{K}(x-y)Q(y)|f(\tilde{u})|dy.
\end{aligned}
$$
By $(K_2)$, 
\begin{equation} \label{A19}
\begin{aligned}
\left(\int_{-\infty}^{x-\frac{1}{\delta}} +\int_{x+\frac{1}{\delta}}^{+\infty} \right)\mathcal{K}(x-y)Q(y)|f(\tilde{u})|dy &\leq \|Q\|_{\infty}\|f(u)\|_\infty\left(\int_{-\infty}^{x-\frac{1}{\delta}} + \int_{x+\frac{1}{\delta}}^{+\infty} \right)\mathcal{K}(x-y)dy \\
&\leq \|Q\|_{\infty}C\left(\int_{-\infty}^{x-\frac{1}{\delta}} + \int_{x+\frac{1}{\delta}}^{+\infty} \right)\frac{dy}{|x-y|^{3}}=C_1\delta.
\end{aligned}
\end{equation}
On the other hand, fixing $q \in (1,\infty)$, $q'=\frac{q}{q-1}$ and using $(K_4)$, we obtain by  H\"older inequality   
$$
\begin{aligned}
\int_{x-\frac{1}{\delta}}^{x+\frac{1}{\delta}}\mathcal{K}(x-y)Q(y)f(\tilde{u})dy&\leq \int_{x-\frac{1}{\delta}}^{x+\frac{1}{\delta}} \mathcal{K}(x-y)Q(y)f(\tilde{u})dy\\
&\leq K \left( \int_{x-\frac{1}{\delta}}^{x+\frac{1}{\delta}} \mathcal{K}^q(x-y)dx\right)^{1/q} \left( \int_{x-\frac{1}{\delta}}^{x+\frac{1}{\delta}} |f(\tilde{u})|^{2q'}dy\right)^{1/q'}.
\end{aligned}
$$
As $u\in L^{p}(\R\setminus (a,b))$, we know that  
$$
\|f(\tilde{u})\|_{L^{p}(x-\frac{1}{\delta}, x+\frac{1}{\delta})} \to 0\;\;\mbox{as}\;\;|x|\to +\infty.
$$
Therefore, there are  $T>0$ such that
\begin{equation} \label{A20}
\int_{x-\frac{1}{\delta}}^{x+\frac{1}{\delta}}\mathcal{K}(x-y)Q(y)f(\tilde{u})dy\leq \delta, \quad \forall  |x|\geq T.
\end{equation}
From (\ref{A19}) and (\ref{A20}),
\begin{equation} \label{A21}
\int_{\mathbb{R}}\mathcal{K}(x-y)Q(y)f(\tilde{u})dy\leq C_1\delta +\delta, \quad \forall |x|\geq T.
\end{equation}
Since $\delta$ is arbitrary, the proof is finished. 
\end{proof}

Now we are able to prove the following lemma 
\begin{lemma}\label{decay}
$$
u(x)\to 0\;\;\mbox{as}\;\;|x|\to \infty.
$$
\end{lemma}
\begin{proof}
Let $v$ be the positive solution of the linear problem (\ref{A13}) and $T>0$ such that 
$$
(a,b)\subset (-T,T).
$$ 
Then there exists $C \gg 1$ such that 
$$
V(x) = Cv(x) \geq 1+\|u\|_{L^\infty(\R \setminus (a,b))}, \quad \mbox{for} \quad |x| \leq T.
$$
Moreover, $V$ is solution of the problem 
\begin{equation}\label{A22} 
\frac{1}{2}(-\Delta)^{1/2} V + V = CQ(x)f(\tilde{u}(x))\;\;\mbox{in}\;\;\R
\end{equation} 
and 
\begin{equation}\label{A23}
V(x) \to 0\;\;\mbox{as}\;\;|x| \to \infty.
\end{equation}
Let 
$$
\varphi (x) = \begin{cases}
(u-V)^+(x),&x\in (-T,T)^c\\
0,&x\in (-T,T).
\end{cases}
$$
We claim that 
\begin{equation}\label{A24}
\varphi \equiv 0.
\end{equation}
Assuming for a moment that (\ref{A24}) is true, we have  
$$
u(x)\leq V(x)\;\;\mbox{a.e.}\;\;x\in \R \setminus (-T,T),
$$
and by (\ref{A23}),  
\begin{equation}\label{A25}
u(x) \to 0\;\;\mbox{as}\;\;|x|\to \infty,
\end{equation}
showing the lemma.

\noindent 
{\bf Proof of the claim.} Since $\varphi \in H^{1/2}(\R)$ and $u$ is solution of problem $(P)$, we have 
\begin{equation}\label{A26}
\frac{1}{2}\iint_{\R^{2} \setminus (a,b)^2}\frac{[u(x)-u(y)][\varphi (x) - \varphi (y)]}{|x-y|^{2}}dy dx + \int_{\R \setminus (a,b)} u(x)\varphi (x)dx = \int_{\R \setminus (a,b)} Q(x)f(u)\varphi (x)dx.
\end{equation}
Moreover, since $V$ is solution of (\ref{A22}), we also have  
$$
\frac{1}{2}\iint_{\R^{2}} \frac{[V(x) - V(y)][\varphi (x) - \varphi (y)]}{|x-y|^{2}}dy dx + \int_{\R} V(x)\varphi (x)dx = \int_{\R} CQ(x)f(\tilde{u}(x))\varphi (x)dx.
$$
Recalling that
$$
 \int_{\R} CQ(x)f(\tilde{u}(x))\varphi (x)dx =\int_{\R \setminus (a,b)} CQ(x)f(\tilde{u}(x))\varphi (x)dx= \int_{\R \setminus (a,b)} CQ(x)f(u(x))\varphi (x)dx,
$$
it follows that
\begin{equation}\label{A27}
\frac{1}{2}\iint_{\R^{2}} \frac{[V(x) - V(y)][\varphi (x) - \varphi (y)]}{|x-y|^{2}}dy dx + \int_{\R \setminus (a,b)} V(x)\varphi (x)dx = \int_{\R \setminus (a,b)} CQ(x)f(u(x))\varphi (x)dx.
\end{equation}
Now by subtracting (\ref{A26}) with (\ref{A27}), we find 
$$
\frac{1}{2}\iint_{\R^{2}\setminus (-T,T)^2}\frac{[(u-V)(x) - (u-V)(y)][\varphi (x) - \varphi (y)]}{|x-y|^{2}}dy dx + \int_{\R \setminus (-T,T)} (u-V)(x)\varphi (x)dx \leq 0.
$$
Using the fact that $V(x) \geq u(x)$ for $x \in (-T,T)$, it is easy to check that 
$$
[(u-V)(x) - (u-V)(y)][\varphi (x) - \varphi (y)]\geq 0, \quad (x,y) \in \R^{2}\setminus (-T,T)^2.
$$
Thus, as $(\R\setminus (-T,T))^{2} \subset \R^{2}\setminus (-T,T)^2$, we get
$$
\frac{1}{2}\iint_{(\R\setminus (-T,T))^{2}} \frac{|(u-V)^{+}(x) - (u-V)^+(y)|^2}{|x-y|^{2}}dy dx + \int_{\R\setminus (-T,T)} [(u-V)^+]^2(x)dx \leq 0,
$$
leading to $(u-V)^+ \equiv 0$.
\end{proof}

\begin{lemma}\label{decay2}
There exists $C>0$ such that 
$$
0\leq u(x) \leq \frac{C}{|x|^{2}}, \quad \forall  x \in \mathbb{R}\setminus \{0\}.
$$
\end{lemma}

\begin{proof}
Arguing as in \cite[Lemma 4.3]{PFAQJT}, it is possible to prove that there is a smooth function $w$ in $\R$ satisfying  
\begin{equation}\label{A28}
\frac{1}{2}(-\Delta)^{1/2}w(x) + \frac{1}{2}w(x) \geq  0 \quad \mbox{for}\;\;|x|>T
\end{equation} 
in the classical sense, where $R$ is fixed of a way such that $ (a,b) \subset (-T,T)$, and   
\begin{equation}\label{A29}
0< w(x) \leq \frac{k_1}{|x|^{2}}, \quad \forall x \in \mathbb{R} \setminus \{0\},
\end{equation}
Note that (\ref{A28}) is equivalent to
\begin{equation}\label{A30}
\frac{1}{2}\iint_{\R^{2}} \frac{[w(x)-w(y)][\phi(x)-\phi(y)]}{|x-y|^{2}}dy dx + \frac{1}{2}\int_{\R} w(x)\phi(x)dx \geq  0,
\end{equation}
for all $\phi \in H^{1/2}(\R)$ with $\phi \geq 0$ and $supp \phi \subset (-T,T)^c$. Without loss of generality, we can assume that 
$$
w(x) \geq 1+\|u\|_{L^{\infty}(\R\setminus \Omega)} \;\;\mbox{for}\;\;|x| \leq T.
$$
Note that, by (\ref{A25}), there is $T>0$ large enough such that 
\begin{equation}\label{A31}
u(x)\left( Q(x)\frac{f(u(x))}{u(x)} - \frac{1}{2}\right) \leq 0,\;\;\mbox{for}\;\;|x|\geq T.
\end{equation} 
As in the last lemma, considering the function 
$$
\varphi (x) = \begin{cases}
(u-w)^+(x),&x\in \R \setminus (-T,T)\\
0,&x\in (-T,T).
\end{cases}
$$ 
it follows from  (\ref{A31}), 
\begin{equation}\label{A32}
\frac{1}{2}\iint_{\R^{2}\setminus (-T,T)^2} \frac{[u(x) - u(y)][\varphi (x) -\varphi (y)]}{|x-y|^{2}}dy dx + \frac{1}{2}\int_{\R \setminus (-T,T)} u(x)\varphi (x)dx\leq 0.
\end{equation} 
Therefore, from (\ref{A30}) and (\ref{A32}),   
\begin{equation}\label{A27}
\frac{1}{2}\iint_{\R^{2}\setminus (-T,T)^2} \frac{[(u-w)(x) - (u-w)(y)][\varphi (x) - \varphi (y)]}{|x-y|^{2}}dy dx + \frac{1}{2}\int_{\R\setminus (-T,T)} (u-w)(x)\varphi (x)dx \leq 0.
\end{equation}
Arguing as in Lemma \ref{decay}, we find
$$
\frac{1}{2}\iint_{(\R\setminus (-T,T))^{2}} \frac{|(u-w)^+(x) - (u-w)^+(y)|^2}{|x-y|^{2}}dy dx + \frac{1}{2}\int_{\R \setminus (-T,T)} [(u-w)^+(x)]^2dx \leq 0.
$$
that is, $(u-w)^+ \equiv 0$. Therefore,   
\begin{equation}\label{A28}
u(x) \leq w(x) \leq \frac{k_2}{|x|^{2}}\;\;\mbox{for all}\;\;x\in \R \setminus (-T,T).
\end{equation}
Now, the result follows by using the fact that $u \in L^{\infty}(\mathbb{R})$.
\end{proof}


\end{document}